\documentclass[11pt]{article}

\RequirePackage{fix-cm}

\usepackage[margin=1in]{geometry}
\usepackage{textcomp}
\usepackage{float}
\usepackage{amsmath}
\usepackage{amssymb}  
\usepackage{epsfig}
\usepackage{amsthm}   
\usepackage[mathcal]{eucal}
\usepackage{mathrsfs}
\usepackage{url}

\usepackage{color}

\usepackage{mathtools}

\newtheorem{algorithm}{Algorithm}

\newcommand{\HH}{\mathcal{H}}

\newcommand{\R}{\mathbb{R}}
\newcommand{\BR}{ {\overline{\mathbb{R}} }  }

\newcommand{\inner}[2]{\langle{#1},{#2}\rangle}

\newcommand{\norm}[1]{\|#1\|}

\newcommand{\tos}{\rightrightarrows} 

\newcommand{\bi}{\begin{itemize}}
\newcommand{\ei}{\end{itemize}}
\newcommand{\ba}{\begin{array}}
\newcommand{\ea}{\end{array}}

\newcommand{\mgap}{\vspace{.1in}}

\newtheorem{theorem}{Theorem}[section]
\newtheorem{lemma}[theorem]{Lemma}

\newtheorem{proposition}[theorem]{Proposition}
\newtheorem{remark}[theorem]{Remark}

\makeatletter
\renewcommand*{\@biblabel}[1]{\hfill#1.}
\makeatother

\begin{document}

\title{ An Inexact Spingarn's Partial Inverse Method with Applications to Operator Splitting
and Composite Optimization}

\author{
    S. Costa Lima
    \thanks{ Departamento de Matem\'atica, Universidade Federal de Santa Catarina, Florian\'opolis, Brazil, 88040-900. Tel.: +55 48 37216560. (E-mail: {\tt samaraclim@outlook.com}). 
    }
     \and
   M. Marques Alves 
    \thanks{ Departamento de Matem\'atica, Universidade Federal de Santa Catarina, Florian\'opolis, Brazil, 88040-900. Tel.: +55 48 37213678  Fax:  +55 48 37216560. (E-mail:  {\tt maicon.alves@ufsc.br}). }  
}
\date{}
\maketitle
\begin{abstract}
   We propose and study the iteration-complexity of an inexact version of the Spingarn's partial inverse
method. Its complexity analysis is performed by viewing it
in the framework of the hybrid proximal extragradient (HPE) method, for which pointwise and ergodic iteration-complexity
has been established recently by Monteiro and Svaiter. As applications, we propose and 
analyze the iteration-complexity of an inexact operator splitting algorithm -- which generalizes the
original Spingarn's splitting method -- and of a parallel forward-backward algorithm for multi-term
composite convex optimization.

\vspace{0.5cm}


\end{abstract}

 \bigskip

 \bigskip
\section{Introduction}
\label{sec:int}
In~\cite{spi-par.amo83}, J. E. Spingarn proposed and analyzed a proximal point type method -- called
the partial inverse method -- for solving the problem of finding a point in the graph of a maximal monotone
operator such that the first (primal) variable belongs to
a closed subspace and the second (dual) variable belongs to its orthogonal
complement. This problem encompasses 
%
%
minimization of convex functions over closed subspaces and inclusion problems given by the
sum of finitely many maximal monotone operators. Regarding the latter case, Spingarn
also derived an operator splitting method with the distinctive feature of allowing
parallel implementations. Spingarn's approach for solving the above mentioned
problem consists in recasting it as an inclusion problem for the partial inverse (a concept
coined by himself) of the monotone operator involved in the formulation of the problem
with respect to the closed subspace. That said, Spingarn's partial inverse method
essentially consists of Rockafellar's proximal point method (PPM) applied to this
monotone inclusion, which converges either under the assumption of exact computation of the resolvent
or under summable error criterion~\cite{roc-mon.sjco76}. The hybrid proximal extragradient
(HPE) method of Solodov and Svaiter~\cite{sol.sva-hyb.svva99} is an inexact version of the
Rockafellar's PPM which uses relative error tolerance criterion for solving each
proximal subproblem instead of summable error condition. The HPE method has been
used for many authors 
\cite{sol.sva-hyb.svva99,bot.cse-hyb.nfao15,cen.mor.yao-hyb.jota10,eck.sil-pra.mp13,he.mon-acc.siam16,ius.sos-pro.opt10,lol.par.sol-cla.jca09,mon.ort.sva-imp.coap14,mon.ort.sva-ada.coap16,MonSva10-2,sol.sva-ine.mor00,Sol-Sv:hy.unif}
as a framework for the design and analysis of several algorithms
for monotone inclusion problems, variational inequalities, saddle-point problems and convex optimization.
Its iteration-complexity has been established recently by Monteiro and Svaiter~\cite{mon.sva-hpe.siam10} and, as 
a consequence, it has proved the iteration-complexity of various important algorithms in optimization (which use the HPE method
as a framework) including Tseng's forward-backward method, Korpelevich extragradient method and
the alternating direction method of multipliers (ADMM)~\cite{MonSva10-2,mon.sva-hpe.siam10,gon.mel.mon-imp.pre16}.

In this paper, we propose and analyze the iteration-complexity of an inexact version of the
Spingarn's partial inverse method in the light of the recent developments in the iteration-complexity of the HPE method.
We introduce a notion of approximate solution of the above mentioned Spingarn's problem 
and prove that our proposed method can be regarded as a special instance of the HPE method
applied to this problem and, as a consequence, we obtain
pointwise and ergodic iteration-complexity results for our inexact partial inverse method.
As applications, we propose and study the iteration-complexity of an inexact operator splitting
method for solving monotone inclusions with the sum of finitely many maximal monotone
operators as well as of a parallel forward-backward algorithm for multi-term composite
convex optimization.
We also briefly discuss how a different inexact version 
of the Spingarn's partial inverse method proposed and studied
in~\cite{bur.sag.sch-ine.oms06} is related  to our method.

\vspace{0.5cm}
\noindent
{\bf Contents.} Section \ref{sec:bm} contains two subsections. Subsection \ref{sec:gn}
presents some general results and the basic notation we need in this paper.
Subsection \ref{sec:pp} is devoted to present the iteration-complexity
of the HPE method and to briefly discuss the method of \cite{bur.sag.sch-ine.oms06}.
Section \ref{sec:pi} presents our main algorithms and its iteration-complexity.
Finally, in Section \ref{sec:spin} we show how the results of Section \ref{sec:pi}
can be used to derive an operator splitting method and a parallel forward-backward
method for solving multi-term composite convex optimization.

\section{Background Materials and Notation}
\label{sec:bm}
This section contains two subsections. In Subsection \ref{sec:gn}
we present the general notation as well as some basic facts about
maximal monotone operators and convex analysis. In Subsection \ref{sec:pp}
we review some important facts about the iteration-complexity
of the hybrid proximal extragradient (HPE) method 
and study some properties of a variant of it.

\subsection{General Results and Notation}
\label{sec:gn}
We denote by $\HH$ a real Hilbert space with
inner product $\inner{\cdot}{\cdot}$ and induced norm $\|\cdot\|:=\sqrt{\inner{\cdot}{\cdot}\textbf{}}$.
For $m\geq 2$, the Hilbert space $\HH^m:=\HH\times \HH\times \cdots \times \HH$
will be endowed with the inner product $\inner{(x_1,\dots, x_m)}{(x_1',\dots, x_m')}:=\sum_{i=1}^m\,\inner{x_i}{x_i'}$
and norm $\|\cdot\|:=\sqrt{\inner{\cdot}{\cdot}\textbf{}}$.

For a set-valued map $S:\HH\tos \HH$, its \emph{graph}
and \emph{domain} are taken respectively as 
$\mbox{G}(S)=\{(x,v)\in \HH\times \HH\,:\, v\in S(x)\}$ and $\mbox{D}(S)=\{x\in \HH\,:\, S(x)\neq \emptyset\}$.
The \emph{inverse} of $S$ is $S^{-1}:\HH\tos \HH$ such that $v\in S(x)$ if and only if $x\in S^{-1}(v)$.
Given $S,S':\HH\tos \HH$ and $\lambda>0$ we define
$S+S':\HH\tos \HH$ and $\lambda S:\HH\tos \HH$
by $(S+S')(x)=S(x)+S'(x)$ and $(\lambda S)(x)=\lambda S(x)$
for all $x\in \HH$, respectively. 
Given set-valued maps $S_i:\HH\tos \HH$, for $i=1,\dots, m$,
we define its product by
\begin{align}
 \label{eq:def.prs}
 S_1\times S_2\times \dots \times S_m:\HH^m\tos \HH^m,
 \quad (x_1,x_2,\dots, x_m)\mapsto S_1(x_1)\times S_2(x_2)\times \cdots \times S_m(x_m).
\end{align}

An operator $T:\HH\tos \HH$ is \emph{monotone} if
\[
 \inner{v-v'}{x-x'}\geq 0\;\;\;\mbox{whenever}\;\;\; (x,v),(x',v')\in \mbox{G}(T).
\]
It is \emph{maximal monotone} if it is monotone and maximal in the following sense:
if $S:\HH\tos \HH$ is monotone and $\mbox{G}(T)\subset \mbox{G}(S)$, then $T=S$.
The \emph{resolvent} of a maximal monotone operator $T$
is $(T+I)^{-1}$, and $\tilde z=(T+I)^{-1}z$ if and only if $z-\tilde z\in T(\tilde z)$.
%
%
For $T:\HH\tos\HH$ maximal monotone and $\varepsilon\geq 0$, the $\varepsilon$-enlargement of 
$T$~\cite{bur.ius.sva-enl.svva97,leg.the-sub.svva96}
is the operator $T^{\varepsilon}:\HH\tos\HH$ defined by
\begin{align}
 \label{eq:def.teps}
 T^{\varepsilon}(x):=\{v\in \HH\;\;:\;\;\inner{v'-v}{x'-x}\geq -\varepsilon\;\;\forall (x',v')\in \mbox{G}(T)\}\quad \forall x\in \HH.
\end{align}
Note that $T(x)\subset T^{\varepsilon}(x)$ for all $x\in \HH$.

The following summarizes some useful properties of $T^{\varepsilon}$ which will be useful
in this paper.
\begin{proposition}
\label{pr:teps}
Let $T, S:\HH\tos \HH$ be set-valued maps.  Then,
\begin{itemize}
\item[\emph{(a)}] if $\varepsilon \leq \varepsilon'$, then
$T^{\varepsilon}(x)\subseteq T^{\varepsilon'}(x)$ for every $x \in \HH$;
\item[\emph{(b)}] $T^{\varepsilon}(x)+S^{\,\varepsilon'}(x) \subseteq
(T+S)^{\varepsilon+\varepsilon'}(x)$ for every $x \in \HH$ and
$\varepsilon, \varepsilon'\geq 0$;
\item[\emph{(c)}] $T$ is monotone if, and only if, $T  \subseteq T^{0}$;
\item[\emph{(d)}] $T$ is maximal monotone if, and only if, $T = T^{0}$;
\item [\emph{(e)}] if $f:X\to\overline{\R}:=\R\cup\{-\infty,+\infty\}$ is proper, convex and closed, then
  $\partial_\varepsilon f(x)\subseteq (\partial f)^{\varepsilon}(x)$ for
  any $\varepsilon \geq 0$ and $x\in \HH$.
\end{itemize}
\end{proposition}
Throughout this work we adopt standard notation of convex analysis for subdiferentials, 
$\varepsilon$-subdiferentials, etc. 
Moreover, for a closed subspace $V\subseteq \HH$
we denote by $V^\perp$ its \emph{orthogonal complement}
and by $P_V$ and $P_{V^\perp}$ the \emph{orthogonal
projectors} onto $V$ and $V^\perp$, respectively.
The \emph{Spingarn's partial inverse}~\cite{spi-par.amo83} of a set-valued map $S:\HH\tos \HH$
with respect to a closed subspace $V$ of $\HH$ is the set-valued
operator $S_V:\HH\tos \HH$ whose graph is
\begin{align}
 \label{eq:def.ig}
 \mbox{G}(S_V):=\{(z,v)\in \HH\times \HH\;:\; P_V(v)+P_{V^\perp}(z)\in S(P_V(z)+P_{V^\perp}(v))\}.
\end{align}
The following lemma will be important for us.

\begin{lemma}\emph{(\cite[Lemma 3.1]{bur.sag.sch-ine.oms06})}
 \label{lm:bss}
 Let $T:\HH\tos \HH$ be a maximal monotone operator, $V\subset \HH$ a closed subspace and $\varepsilon>0$. Then,
\begin{align*}
  (T_V)^\varepsilon=(T^{\varepsilon})_V.
\end{align*}
\end{lemma}

Next we present the transportation formula
for $\varepsilon$-enlargements.

\begin{theorem}\emph{(\cite[Theorem 2.3]{bur.sag.sva-enl.col99})}
 \label{th:tf}
  Suppose $T:\HH\tos \HH$ is maximal monotone and
	let $x_\ell, u_\ell\in \HH$, $\varepsilon_\ell, \alpha_\ell\in \R_+$,
	for $\ell=1,\dots, k$, be such that
	 \[
	 u_\ell\in T^{\varepsilon_\ell}(x_\ell),\quad \ell=1,\dots, k,\quad  \sum_{\ell=1}^k\,\alpha_\ell=1,
	\]
	and define
	\[
	 x^a:=\sum_{\ell=1}^k\,\alpha_\ell\, x_\ell\,,\quad u^a:=\sum_{\ell=1}^k\,\alpha_\ell\; u_\ell\,,\quad
	 \varepsilon^a:=\sum_{\ell=1}^k\,\alpha_\ell \left[\varepsilon_\ell+\inner{x_\ell-x^a}{u_\ell-u^a}\right].
	\]
	 Then, the following statements hold:
	 \begin{itemize}
	 \item[\emph{(a)}] $\varepsilon^a\geq 0$ and $u^a\in T^{\varepsilon^a}(x^a)$.
	  \item[\emph{(b)}] If, in addition, $T=\partial f$ for some proper, convex and closed function
		$f$ and $u_\ell\in \partial_{\varepsilon_\ell} f(x_{\ell})$ for $\ell=1,\dots, k$,
		then $u^a\in \partial_{\varepsilon^a} f(x^a)$.
	\end{itemize}
\end{theorem}

The following results will also be useful in this work.

\begin{lemma}\emph{(\cite[Lemma 3.2]{mon.sva-icc.pre})}
 \label{lm:ms}
 If $f:\HH\to \BR$ is a proper, closed and convex function and $x,\tilde x, v\in \HH$
are such that $v\in \partial f(x)$ and $f(\tilde x)<\infty$, then $v\in \partial_\varepsilon f(\tilde x)$
for every $\varepsilon\geq f(\tilde x)-f(x)-\inner{v}{\tilde x-x}$. 
\end{lemma}


\begin{lemma}
\label{lm:crules}
 Let $f:\HH\to \BR$ be proper, closed and convex. Then, the following holds
for every $\lambda, \varepsilon>0$:
 \begin{itemize}
 \item[\emph{(a)}] $\partial (\lambda f)=\lambda \partial f$;
 \item[\emph{(b)}] $\partial_\varepsilon (\lambda f)=\lambda \partial_{\varepsilon/\lambda} f$.
 \end{itemize}
\end{lemma}

\begin{lemma}\emph{(\cite[Lemmas 1.2.3]{nes-int.book})}
\label{lm:ineq.lin}
Let $f:\HH\to \R$ be convex and continuously differentiable
such that there exists a nonnegative constant $L$ satisfying
\[
 \norm{\nabla f(x)-\nabla f(y)}\leq L\norm{x-y}\quad \forall x,y\in \HH\,.
\]
Then,
\begin{equation}
\label{eq:ineq.lin}
 0\leq f(x)-f(y)-\inner{\nabla f(y)}{x-y}\leq \dfrac{L}{2}\norm{x-y}^2
\end{equation}
for all $x,\tilde x\in \HH$.
\end{lemma}

\subsection{On the Hybrid Proximal Extragradient Method}
\label{sec:pp}

In this subsection we consider the \emph{monotone inclusion problem}
\begin{align}
\label{eq:cp}
 \begin{aligned}
  0 \in T(z)
 \end{aligned}
\end{align}
where $T:\HH\tos \HH$ is a point-to-set
maximal monotone operator. We also assume
that the solution set $T^{-1}(0)$ of \eqref{eq:cp} is nonempty.  
Since problem \eqref{eq:cp} appears in different fields of applied mathematics
including optimization, equilibrium theory and partial differential equations, it is
desirable to have efficient numerical schemes to find approximate solutions of it.

An \emph{exact proximal point method} (PPM) iteration
for \eqref{eq:cp} is
\begin{align}
 \label{eq:exact.prox}
 z_k= (\lambda_k T+I)^{-1}z_{k-1}
\end{align}
where $z_{k-1}$ and $z_k$ are the current and new iterate, respectively, and $\lambda_k>0$ is
a sequence of stepsizes. The practical applicability of proximal point algorithms to concrete
problems depends on the availability of inexact versions
of such methods.  
In the seminal work \cite{roc-mon.sjco76}, Rockafellar proved that 
if $z_k$ is computed satisfying
\begin{align}
 \label{eq:inexact.prox}
 \norm{z_k-(\lambda_k T+I)^{-1}z_{k-1}}\leq \eta_k,\quad \sum_{k=1}^\infty\,\eta_k<\infty,
\end{align}
and $\{\lambda_k\}$ is bounded away from zero, then $\{z_k\}$ converges weakly to a solution 
of \eqref{eq:cp} -- assuming that there exists at least one of them.
New inexact versions of the PPM which use \emph{relative error tolerance} to compute approximate solutions
have been proposed and intensively studied in the last two decades 
\cite{sol.sva-hyb.svva99,MonSva10-2,sol.sva-ine.mor00,Sol-Sv:hy.unif,mon.sva-hpe.siam10,MonSva10-1,sol.sva-hyb.jca99}. The key idea behind
such methods~\cite{sol.sva-hyb.svva99} consists in decoupling \eqref{eq:exact.prox}
as an inclusion and an equation:
\begin{align}
  \label{eq:dec.prox}
  v_k\in T(z_k),\quad  \lambda_k v_k+z_k-z_{k-1}=0
\end{align} 
and in relaxing both of them according to relative error tolerance criteria.
The \emph{hybrid proximal extragradient} (HPE) 
method of \cite{sol.sva-hyb.svva99} has been used in the last few years as a framework for the analysis and development
of several algorithms for solving monotone inclusion, saddle-point and convex optimization
problems~\cite{sol.sva-hyb.svva99,bot.cse-hyb.nfao15,cen.mor.yao-hyb.jota10,eck.sil-pra.mp13,he.mon-acc.siam16,ius.sos-pro.opt10,lol.par.sol-cla.jca09,mon.ort.sva-imp.coap14,mon.ort.sva-ada.coap16,MonSva10-2,sol.sva-ine.mor00,Sol-Sv:hy.unif}. 

Next we present the HPE method.

\mgap
\mgap

\noindent
\fbox{
\addtolength{\linewidth}{-2\fboxsep}%
\addtolength{\linewidth}{-2\fboxrule}%
\begin{minipage}{\linewidth}
\begin{algorithm}
\label{hpe}
{\bf Hybrid proximal extragradient (HPE) method for \bf{(\ref{eq:cp})}}
\end{algorithm}
\begin{itemize}
\item[(0)] Let $z_0\in \HH$ and $\sigma\in [0,1[$ be given and set $k=1$.
\item [(1)] Compute $(\tilde z_k,v_k,\varepsilon_k)\in \HH\times \HH\times \R_+$
and $\lambda_k>0$ such that
\begin{align}
\label{eq:hpe}
 \begin{aligned}
  v_k\in T^{\varepsilon_k}(\tilde z_k),\quad  \norm{\lambda_kv_k+\tilde z_k-z_{k-1}}^2+
 2\lambda_k\varepsilon_k \leq \sigma^2\norm{\tilde z_k-z_{k-1}}^2.
\end{aligned}
\end{align} 
\item[(2)] Define 
 \begin{align}
  \label{eq:hpe2}
    z_k=z_{k-1}-\lambda_k v_k, 
   \end{align}
	set $k\leftarrow k+1$ and go to step 1.
   \end{itemize}
\noindent
\end{minipage}
} 
\mgap
\mgap

\noindent
\emph{Remarks.} 1) First note that condition \eqref{eq:hpe}
relaxes both the inclusion and the
equation in \eqref{eq:dec.prox}. Here, $T^{\varepsilon}(\cdot)$
is the $\varepsilon$-enlargement of $T$; it has the property that
$T(z)\subset T^{\varepsilon}(z)$ (see Subsection \ref{sec:gn} for details).
2) Instead of $\tilde z_k$, the next iterate $z_k$ is defined in \eqref{eq:hpe2}
as an extragradient step from $z_{k-1}$. 
3) Letting $\sigma=0$ and using Proposition \ref{pr:teps}(d)
we conclude from \eqref{eq:hpe} and \eqref{eq:hpe2} that $(z_k,v_k)$
and $\lambda_k>0$ satisfy \eqref{eq:dec.prox}, i.e., Algorithm \ref{hpe}
is an inexact version of the exact Rockafellar's PPM.
4) Algorithm \ref{hpe} serves also as a framework
for the analysis and development of several numerical schemes for solving concret instances
of \eqref{eq:cp}
(see, e.g., \cite{eck.sil-pra.mp13,mon.ort.sva-imp.coap14,mon.ort.sva-ada.coap16,MonSva10-2,mon.sva-hpe.siam10,MonSva10-1}); specific strategies for computing $(\tilde z_k,v_k,\varepsilon_k)$
and $\lambda_k>0$ satisfying \eqref{eq:hpe} depends on the particular instance of
\eqref{eq:cp} under consideration.
%

%


\mgap

In the last few years, starting with the paper \cite{mon.sva-hpe.siam10}, a lot of
research has been done to study and analyze the \emph{iteration-complexity} of
the HPE method and its special instances, including Tseng's forward-backward
splitting method, Korpelevich extragradient method, ADMM, etc~\cite{MonSva10-2,mon.sva-hpe.siam10,MonSva10-1}.
These iteration-complexity bounds for the HPE method are based on the following termination
criterion introduced in~\cite{mon.sva-hpe.siam10}: for given tolerances $\rho, \epsilon>0$, find
$\bar z,\bar v\in \HH$ and $\bar\varepsilon>0$ such that
$(z,v):=(\bar z,\bar v)$ and $\varepsilon:=\bar \varepsilon$ satisfy
\begin{align}
  \label{eq:tc}
	 v\in T^{\varepsilon}(z),\quad \norm{v}\leq \rho, \quad \varepsilon\leq \epsilon. 
\end{align}
Using Proposition \ref{pr:teps}(d) we find that if $\rho=\epsilon=0$ in \eqref{eq:tc} then $0\in T(\bar z)$, i.e., $\bar z$ is a solution 
of \eqref{eq:cp}.

Next we summarize the main results from \cite{mon.sva-hpe.siam10}
about \emph{pointwise} and 
\emph{ergodic} iteration-complexity
of the HPE method that we will need in this paper.
The \emph{aggregate stepsize sequence} 
$\{\Lambda_k\}$ and the \emph{ergodic sequences}
$\{\tilde {z}_k^a\}$, $\{\tilde v_k^a\}$, 
$\{\varepsilon_k^a\}$ associated
to $\{\lambda_k\}$ and
$\{\tilde {z}_k\}$,  $\{v_k\}$, and
$\{\varepsilon_k\}$ are, respectively,
\begin{align}
\label{eq:d.eg}
  \begin{aligned}
    &\Lambda_k:=\sum_{\ell=1}^k\, \lambda_\ell\,,\\
    &\tilde {z}_k^{\,a}:= \frac{1}{\;\Lambda_k}\;
   \sum_{\ell=1}^k\,\lambda_\ell\, \tilde {z}_\ell, \quad 
   v_k^{\,a}:= \frac{1}{\;\Lambda_k}\;\sum_{\ell=1}^k\, \lambda_\ell\, v_\ell,\\ 
   &\varepsilon_k^{\,a}:=
    \frac{1}{\;\Lambda_k}\;\sum_{\ell=1}^k\,\lambda_\ell (\varepsilon_\ell
    +\inner{\tilde {z}_\ell-
     \tilde {z}_k^{\,a}}{v_\ell-v_k^{\,a}}).
  \end{aligned}
\end{align}

\begin{theorem}[{\cite[Theorem 4.4(a) and 4.7]{mon.sva-hpe.siam10}}]
  \label{lm:rhpe2}
	Let $\{\tilde z_k\}$, $\{v_k\}$, etc be generated by \emph{Algorithm \ref{hpe}}
	and let $\{\tilde z_k^a\}$, $\{v_k^a\}$, etc be given in \eqref{eq:d.eg}. 
  Let also $d_0$ denote the distance of $z_0$ to $T^{-1}(0)\neq\emptyset$
	and assume that $\underline{\lambda}:=\inf \lambda_k>0$. The following
	statements hold.
  \begin{enumerate}
  \item[\emph{(a)}] For any $k\geq 1$,
	  there exists $i\in\{1,\dots,k\}$ 
    such that  
		\begin{align*}
      v_i\in T^{\varepsilon_i}(\tilde z_i),\quad  \norm{v_i}\leq \dfrac{d_0}{\underline{\lambda}\sqrt{k}}
			\sqrt{\dfrac{1+\sigma}{1-\sigma}},\quad  \varepsilon_i\leq \dfrac{\sigma^2 d_0^2}{2(1-\sigma^2)\underline{\lambda}k}\;;
    \end{align*}
	 \item[\emph{(b)}]
	  for any $k\geq 1$,
		 \begin{align*}
     v_k^a\in T^{\varepsilon_k^a}(\tilde z_k^a),\quad 
     \norm{v_k^a}\leq \dfrac{2d_0}{\underline{\lambda}k},\quad
     \varepsilon_k^a\leq \dfrac{2(1+\sigma/\sqrt{1-\sigma^2})d_0^2}{\underline{\lambda}k}\,.
    \end{align*}
	\end{enumerate}
\end{theorem}

\noindent
\begin{remark}
The bounds given in (a) and (b) of Theorem \ref{lm:rhpe2} are
called pointwise and ergodic bounds, respectively. Items (a) and (b) can be used, respectively,
to prove that for given tolerances $\rho,\epsilon>0$ 
the termination criterion \eqref{eq:tc} is satisfied 
in at most
 \begin{align*}
 \mathcal{O}\left(\max\left\{\left\lceil\dfrac{d_0^2}{\underline{\lambda}^2\rho^2}\right\rceil,
 \left\lceil\dfrac{d_0^2}{\underline{\lambda}\epsilon}\right\rceil\right\}\right)
 \;\;\mbox{and}\;\;
\mathcal{O}\left(\max\left\{\left\lceil\dfrac{d_0}{\underline{\lambda}\rho}\right\rceil,\left\lceil\dfrac{d_0^2}{\underline{\lambda}\epsilon}
\right\rceil
\right\}\right)
 \end{align*}
iterations, respectively.
\end{remark}



The following variant of Algorithm \ref{hpe} studied
in \cite{bur.sag.sch-ine.oms06} is related to the results of this paper:
Let $z_0\in \HH$ and $\hat \sigma\in [0,1[$ be given and iterate for $k\geq 1$,

\begin{align}
\left\{
       \begin{array}{ll}
			 \label{eq:v.hpe}
			  \vspace{0.3cm}
				 v_k\in T^{\varepsilon_k}(\tilde z_k),\quad  \norm{\lambda_k v_k+\tilde z_k-z_{k-1}}^2+
 2\lambda_k\varepsilon_k \leq \hat \sigma^2\left(\norm{\tilde z_k-z_{k-1}}^2+\norm{\lambda_kv_k}^2\right),\\
z_k=z_{k-1}-\lambda_k v_k.\\
        \end{array}
        \right.
\end{align}


%
%
 %
%
%
%

\begin{remark} The inequality in \eqref{eq:v.hpe} is a relative error tolerance proposed in \cite{Sol-Sv:hy.unif}
 (for a different method); the identity in \eqref{eq:v.hpe} is the same extragradient step 
of Algorithm \ref{hpe}. Hence, the method described in \eqref{eq:v.hpe}
can be interpreted as a HPE variant in which a different relative error tolerance is
considered in the solution of each subproblem. In what follows in this section we will show that 
\eqref{eq:v.hpe} is actually a special instance of Algorithm \ref{hpe} whenever $\hat\sigma\in [0,1/\sqrt{5}\,[$ and that
it may fail to converge if we take $\hat\sigma>1/\sqrt{5}$.
\end{remark}

\begin{lemma}\emph{(\cite[Lemma 2]{Sol-Sv:hy.unif})}
\label{lm:ineq.2}
 Suppose $\{z_k\}$, $\{\tilde z_k\}$, $\{v_k\}$ and $\{\lambda_k\}$
 satisfy the inequality in \eqref{eq:v.hpe}. 
Then, for every $k\geq 1$,
\begin{align*}
  \dfrac{1-\theta}{1-\hat \sigma^2}\norm{\tilde z_k-z_{k-1}}\leq \norm{\lambda_k v_k}\leq
	\dfrac{1+\theta}{1-\hat \sigma^2}\norm{\tilde z_k-z_{k-1}}
\end{align*}
where
\begin{align}
 \label{eq:lm:ineq.20}
 \theta:=\sqrt{1-(1- \hat \sigma^2)^2}.
\end{align}
\end{lemma}
\begin{proof}
 From the inequality in \eqref{eq:v.hpe} and the Cauchy-Schwarz inequality we obtain
 \begin{equation*}
  \left(1- \hat \sigma^2\right)\norm{\lambda_k v_k}^2-2\norm{\tilde z_k-z_{k-1}}\norm{\lambda_k v_k}
	+\left(1-\hat \sigma^2\right)\norm{\tilde z_k-z_{k-1}}^2\leq 0,\qquad \forall k\geq 1.
\end{equation*}
 To finish the proof of the lemma note that (in the above inequality) 
we have a quadratic function in the term $\norm{\lambda_k v_k}$.
\qed
\end{proof}

\begin{proposition}
 \label{pr:v.hpe}
 Let $\{z_k\}$, $\{\tilde z_k\}$, $\{v_k\}$, $\{\varepsilon_k\}$ and $\{\lambda_k\}$
 be given in \eqref{eq:v.hpe} and 
 assume that $\hat\sigma\in [0,1/\sqrt{5}\,[$.
 Define, for all $k\geq 1$,
\begin{align}
 \label{eq:pr:v.hpe0}
  \sigma:=\hat \sigma\sqrt{1+\left(\dfrac{1+\theta}{1-\hat\sigma^2}\right)^2},
\end{align}
where $0\leq \theta<1$ is given in \eqref{eq:lm:ineq.20}. Then, $\sigma\geq 0$ belongs to $[0,1[$ and $z_k, \tilde z_k, v_k, \varepsilon_k$
and $\lambda_k>0$ satisfy \eqref{eq:hpe} and \eqref{eq:hpe2} for all $k\geq 1$. As a consequence, the method of 
\emph{\cite{bur.sag.sch-ine.oms06}} defined in 
\eqref{eq:v.hpe} is a special instance of \emph{Algorithm \ref{hpe}} whenever $\hat\sigma\in [0,1/\sqrt{5}\,[$.
\end{proposition}
\begin{proof}
 The assumption $\hat\sigma\in [0,1/\sqrt{5}\,[$, definition \eqref{eq:pr:v.hpe0}
 and some simple calculations show that $\sigma\in [0,1[$. It follows from
\eqref{eq:v.hpe}, \eqref{eq:hpe} and \eqref{eq:hpe2}
that to finish the proof of the proposition it suffices to 
prove the inequality in \eqref{eq:hpe}. To this end, note
that from the second inequality in Lemma \ref{lm:ineq.2} and \eqref{eq:pr:v.hpe0} we have
\begin{align*}
 \hat\sigma^2\left(\norm{\tilde z_k-z_{k-1}}^2+\norm{\lambda_k v_k}^2\right)&\leq
 \hat\sigma^2\left(1+\left(\dfrac{1+\theta}{1-\hat\sigma^2}\right)^2\right)\norm{\tilde z_k-z_{k-1}}^2\\
  &=\sigma^2\norm{\tilde z_k-z_{k-1}}^2\quad \forall k\geq 1,
\end{align*}
which in turn gives that the inequality in \eqref{eq:hpe} follows from the one in \eqref{eq:v.hpe}. 
\qed
\end{proof}

\begin{remark}
 \label{rm:um.not.dois}
Algorithm \ref{hpe} is obviously a
special instance of \eqref{eq:v.hpe} whenever $\sigma\in [0,1/\sqrt{5}\,[$
by setting $\hat\sigma:=\sigma$. 
Next we will show it is not true in general. Let $T:\R\to \R$ be the maximal monotone operator
defined by
\begin{align}
\label{eq:def.Tr2}
 T(z):=z\quad \forall z\in \R.
\end{align}
Assume that $\sigma\in ]\sqrt{2/5},1[$, take $z_0=1$ and define, for all $k\geq 1$,
\begin{align}
 \label{eq:def.algs}
  \tilde z_k:=z_k:=\left(1-\sigma^2\right)z_{k-1},\quad v_k:=z_{k-1},\quad \varepsilon_k:=\frac{\sigma^4}{2} |z_{k-1}|^2,
	\quad \lambda_k:=\sigma^2.
\end{align}  
We will show that $(\tilde z_k, v_k,\varepsilon_k)$ and $\lambda_k>0$ in \eqref{eq:def.algs} 
satisfy \eqref{eq:hpe} but not \eqref{eq:v.hpe} for any choice of $\hat\sigma\in [0,1/\sqrt{5}\,[$.
To this end, we first claim that $v_k\in T^{\varepsilon_k}(\tilde z_k)$ for all $k\geq 1$.
Indeed, using \eqref{eq:def.Tr2} and \eqref{eq:def.algs} we obtain, for all $y\in \R$ and $k\geq 1$, 
\begin{align*}
 (Ty-v_k)(y-\tilde z_k)&=(y-z_{k-1})(y-z_{k-1}+\sigma^2 z_{k-1})\\
  &\geq |y-z_{k-1}|^2-|\sigma^2 z_{k-1}||y-z_{k-1}| \\
	&\geq -\dfrac{|\sigma^2 z_{k-1}|^2}{4}>-\varepsilon_k,
\end{align*}
which combined with \eqref{eq:def.teps} proves our claim. Moreover, it follows from \eqref{eq:def.algs}
that
\begin{align}
 \nonumber
 |\lambda_k v_k+\tilde z_k-z_{k-1}|^2+2\lambda_k\varepsilon_k&=|\tilde z_k-(1-\sigma^2)z_{k-1}|^2+2\lambda_k\varepsilon_k\\
 \nonumber 
&=2\lambda_k\varepsilon_k\\
 \label{eq:401}
 &=\sigma^2|\tilde z_k-z_{k-1}|^2,
\end{align}
which proves that $(\tilde z_k,v_k,\varepsilon_k)$
and $\lambda_k>0$ satisfy the inequality in \eqref{eq:hpe}. The first and second identities in \eqref{eq:def.algs}
give that they also satisfy \eqref{eq:hpe2}.
Altogether, we have that the iteration defined in \eqref{eq:def.algs}
is generated by Algorithm \ref{hpe} for solving \eqref{eq:cp} with
$T$ given in \eqref{eq:def.Tr2}.
On the other hand, it follows from \eqref{eq:def.algs}
and the assumption $\sigma>\sqrt{2/5}$ that
\begin{align}
 \nonumber
 \sigma^2|\tilde z_k-z_{k-1}|^2&=\dfrac{\sigma^2}{2}\left(|\tilde z_k-z_{k-1}|^2+|\lambda_k v_k|^2\right)\\
  \label{eq:400}
  &>\dfrac{1}{5}\left(|\tilde z_k-z_{k-1}|^2+|\lambda_k v_k|^2\right).
\end{align}
Hence, it follows from \eqref{eq:401} and \eqref{eq:400} that the inequality in \eqref{eq:v.hpe}
can not be satisfied for any choice of $\hat\sigma\in [0,1/\sqrt{5}\,[$ and so the sequence given is 
\eqref{eq:def.algs} is generated by Algorithm \ref{hpe} but it is not generated by 
the algorithm described in \eqref{eq:v.hpe}.
\end{remark}


\begin{remark}
\label{rm:notconv}
Next we present an example of a monotone inclusion problem
for which an instance of \eqref{eq:v.hpe} may fail to converge if we take
$\hat \sigma \in ]1/\sqrt{5},1[$. To this end, consider problem \eqref{eq:cp}
where the maximal monotone operator $T:\R\to \R$ is defined by
\begin{align}
 \label{eq:def.Tr}
  T(z):=\alpha z \qquad \forall z\in \R,
\end{align}
where 
\begin{align}
  \label{eq:sl}
  \alpha:=\dfrac{2\gamma}{\gamma-2}+1,\quad  \gamma:=\dfrac{1+\theta}{1-\hat\sigma^2}.
\end{align}
($\theta>0$ is defined in \eqref{eq:lm:ineq.20}.) Assuming $\hat \sigma \in ]1/\sqrt{5},1[$
we obtain $5\hat\sigma^4-6\hat\sigma^2+1<0$, which is clearly
equivalent to $\theta>|1-2\hat\sigma^2|$. Using \eqref{eq:sl} and the latter inequality
we conclude that
\begin{align}
 \label{eq:sl4}
 \gamma>2,\quad \dfrac{\alpha\gamma}{\alpha+\gamma}>2.
\end{align}
Now take $z_0=1$ and define, for all $k\geq 1$,
\begin{align}
 \label{eq:sl2}
  (\tilde z_k,v_k,\varepsilon_k):=\left(\dfrac{\gamma}{\alpha+\gamma}z_{k-1},T(\tilde z_k),0\right),\quad \lambda_k:=1,\quad
  z_k:=z_{k-1}-\lambda_k v_k.
\end{align}
Direct calculation yields, for all $k\geq 1$,
\begin{align}
 |v_k+\tilde z_k-z_{k-1}|^2=\hat\sigma^2\left(|\tilde z_k-z_{k-1}|^2+|v_k|^2\right),
\end{align}
which, in turn, together with \eqref{eq:sl2} imply
\eqref{eq:v.hpe}. Using \eqref{eq:def.Tr} and \eqref{eq:sl2} we find
\begin{align}
 z_k=\left(1-\dfrac{\alpha\gamma}{\alpha+\gamma}\right)^k,\quad \forall k\geq 1.
\end{align}
Using the second inequality in \eqref{eq:sl4} and the latter identity
we easily conclude that $|z_k|\to \infty$ as $k\to \infty$
and so $\{z_k\}$ does not converge to the unique solution 
$\bar z=0$ of \eqref{eq:cp}. 
%
\end{remark}

\section{An Inexact Spingarn's Partial Inverse Method}
\label{sec:pi}
In this section we consider the problem of finding $x,u\in \HH$ such that
\begin{align}
  \label{eq:inc.v}
  x\in V,\quad  u\in V^\perp\;\;\mbox{and}\;\;u\in T(x) 
\end{align}
where $T:\HH\tos \HH$ is maximal monotone and $V$ is a closed subspace of $\HH$.
We define the solution set of \eqref{eq:inc.v} by
 \begin{equation}
   \label{def:sol.v}
  \mathcal{S}^{*}(V,T):=\{z\in \HH\;:\;\mbox{there exist}\; x,u\in \HH\;
	 \mbox{satisfying (\ref{eq:inc.v}) such that}\; z=x+u\}
\end{equation} 
and assume it is nonempty. Problem \eqref{eq:inc.v} encompasses important problems in applied mathematics including minimization of convex functions over closed
subspaces, splitting methods for the sum of finitely many maximal monotone operators
and decomposition methods in convex optimization~\cite{spi-par.amo83,bur.sag.sch-ine.oms06,com-pri.ol14,ouo-eps.mp04}.
One of our main goals in this paper is to propose and analyze an inexact partial inverse method  for solving \eqref{eq:inc.v}
in the light of recent developments in the iteration-complexity theory of
the HPE method~\cite{sol.sva-hyb.svva99,mon.sva-hpe.siam10}, as discussed in Section \ref{sec:pp}. 
We will show, in particular, that
the method proposed in this section generalizes the inexact versions of the Spingarn's partial inverse method
for solving \eqref{eq:inc.v} proposed in \cite{ouo-eps.mp04} and \cite{bur.sag.sch-ine.oms06}. The main results
of iteration-complexity to find approximate solutions are achieved
by analyzing the proposed method in the framework of the HPE method (Algorithm \ref{hpe}).

Regarding the results of iteration-complexity, we will consider the following notion of approximate solution for \eqref{eq:inc.v}: given
tolerances $\rho, \epsilon>0$, find $\bar x,\bar u\in \HH$ 
and $\bar \varepsilon>0$ such that $(x,u)=(\bar x,\bar u)$
and $\varepsilon=\bar \varepsilon$ satisfy
\begin{align}
 \label{eq:inc.v02}
  u\in T^{\varepsilon}(x),\quad \max\left\{\norm{x-P_V(x)}, \norm{u-P_{V^\perp}(u)}\right\}\leq \rho,
	\quad \varepsilon\leq \epsilon,
\end{align}
where $P_V$ and $P_{V^\perp}$ stand for the orthogonal projection onto
$V$ and $V^\perp$, respectively, and $T^\varepsilon(\cdot)$
denotes the $\varepsilon$-enlargement of $T$ (see Section \ref{sec:pp} for more details on notation).
For $\rho=\epsilon=0$, criterion \eqref{eq:inc.v02} gives
$\bar x\in V$, $\bar u\in V^{\perp}$ and $\bar u\in T(\bar x)$, i.e., in this case
$\bar x, \bar u$ satisfy \eqref{eq:inc.v}. Moreover, if $V=\HH$
in \eqref{eq:inc.v}, in which case $P_V=I$ and $P_{V^\perp}=0$, then
the criterion \eqref{eq:inc.v02} coincides with one discussed in Section \ref{sec:pp}
for problem \eqref{eq:cp} (see \eqref{eq:tc}).

\mgap

That said, we next present our inexact version of the Spingarn's partial inverse method
for solving \eqref{eq:inc.v}.

\mgap
\mgap

\noindent
\fbox{
\addtolength{\linewidth}{-2\fboxsep}%
\addtolength{\linewidth}{-2\fboxrule}%
\begin{minipage}{\linewidth}
\begin{algorithm}
\label{spin}
{\bf An inexact Spingarn's partial inverse method for \bf{(\ref{eq:inc.v})} (I)}
\end{algorithm}
\begin{itemize}
\item[(0)] Let $x_0\in \HH$ and $\sigma\in [0,1[$ be given and set $k=1$.
\item [(1)] Compute $(\tilde x_k,u_k,\varepsilon_k)\in \HH\times \HH\times \R_+$ such that
\begin{align}
\label{eq:spin}
 \begin{aligned}
  u_k\in T^{\varepsilon_k}(\tilde x_k),\quad \norm{u_k+\tilde x_k-x_{k-1}}^2+2\varepsilon_k\leq \sigma^2 
  \norm{P_V(\tilde x_k)+P_{V^\perp}(u_k)-x_{k-1}}^2.
\end{aligned}
\end{align} 
\item[(2)] Define 
 \begin{align}
  \label{eq:spin2}
     x_k=x_{k-1}-\left[P_V(u_k)+P_{V^\perp}(\tilde x_k)\right],
   \end{align}
	set $k\leftarrow k+1$ and go to step 1.
   \end{itemize}
\noindent
\end{minipage}
} 
\mgap
\mgap

\noindent
\emph{Remarks.} 1) Letting $V=\HH$ in \eqref{eq:inc.v}, in which case $P_V=I$ and $P_{V^\perp}=0$,
we obtain that Algorithm \ref{spin} coincides with Algorithm \ref{hpe} with $\lambda_k=1$ for all $k\geq 1$ for solving
\eqref{eq:cp} (or, equivalently, \eqref{eq:inc.v} with $V=\HH$). 
2) An inexact partial inverse method called \emph{sPIM($\varepsilon$)} was proposed 
in \cite{bur.sag.sch-ine.oms06}, Section 4.2, for solving \eqref{eq:inc.v}. The latter
method, with a different notation and scaling factor $\eta=1$,
is given according to the iteration:
\begin{align}
\label{eq:66}
\left\{
       \begin{array}{ll}
			  \vspace{0.3cm}
				 u_k\in T^{\varepsilon_k}(\tilde x_k),\;\;
\norm{u_k+\tilde x_k-x_{k-1}}^2+2\varepsilon_k\leq 
 \hat \sigma^2\left(\norm{\tilde x_k-P_V(x_{k-1})}^2+\norm{u_k-P_{V^\perp}(x_{k-1})}^2\right),\\
x_k=x_{k-1}-\left[P_V(u_k)+P_{V^\perp}(\tilde x_k)\right],\\
        \end{array}
        \right.
\end{align}
where $\hat\sigma\in [0,1[$. 
The convergence analysis given in \cite{bur.sag.sch-ine.oms06} for the iteration
\eqref{eq:66} relies on the fact (proved in the latter reference)
that \eqref{eq:66} is a special instance of \eqref{eq:v.hpe} (which 
we observed in Remark \ref{rm:notconv} may fail to converge if
we consider $\hat\sigma\in ]1/\sqrt{5},1[$\,). 
Using the fact just mentioned, the last statement in Proposition \ref{pr:v.hpe}
and Proposition \ref{pr:pi.e.hpe} we conclude that \eqref{eq:66}
is a special instance of Algorithm \ref{spin} whenever $\hat\sigma \in [0,1/\sqrt{5}[$
and it may fail to converge if $\hat\sigma>1/\sqrt{5}$.
On the other hand, since, due to Proposition \ref{pr:pi.e.hpe}, Algorithm \ref{spin}
is a special instance of Algorithm \ref{hpe},
it converges for all $\sigma\in [0,1[$ (see, e.g., \cite[Theorem 3.1]{sol.sva-hyb.svva99}). 
Note that the difference between sPIM($\varepsilon$)
and Algorithm \ref{spin} is the inequality in \eqref{eq:spin} and \eqref{eq:66}.

In what follows we will prove iteration-complexity 
results for Algorithm \ref{spin} to obtain approximate solutions
of \eqref{eq:inc.v}, according to \eqref{eq:inc.v02}, as a consequence
of the iteration-complexity results from Theorem \ref{lm:rhpe2}.
To this end, first let $\{\tilde x_k\}$, $\{u_k\}$ and $\{\varepsilon_k\}$ be generated by Algorithm \ref{spin}
and define the \emph{ergodic} sequences associated to them:
\begin{align}
  \label{eq:def.erg2}
 \begin{aligned}
 &\tilde x_k^a:=\dfrac{1}{k}\sum_{\ell=1}^k \tilde x_\ell\,,\quad  u_k^a:=\dfrac{1}{k}\sum_{\ell=1}^k u_\ell\,,\\ 
 &\varepsilon_k^{\,a}:=\dfrac{1}{k}\sum_{\ell=1}^k\,
\big[\varepsilon_\ell+\inner{\tilde x_\ell-\tilde x_k^a}{u_\ell-u_k^a}\big].
\end{aligned}
\end{align} 

The proof of the next Proposition is given in Subsection \ref{subsec:pp}.

\begin{theorem}
 \label{pr:rate.spin2}
 Let $\{\tilde x_k\}$, $\{u_k\}$ and $\{\varepsilon_k\}$ be generated by \emph{Algorithm \ref{spin}} and let
 $\{\tilde x_k^a\}$, $\{u_k^a\}$ and $\{\varepsilon_k^{\,a}\}$ be defined in \eqref{eq:def.erg2} . Let also $d_{0,V}$ 
denote the distance of $x_0$ to the solution set \eqref{def:sol.v}.  
 The following statements hold:
\begin{enumerate}
  \item[\emph{(a)}] For any $k\geq 1$,
	  there exists $j\in \{1,\dots,k\}$ 
    such that  
		\begin{align}
		  \label{eq:702}
			\begin{aligned}
      & u_j\in T^{\varepsilon_j}(\tilde x_j),\\  
			&\sqrt{\norm{\tilde x_j-P_{V}(\tilde x_j)}^2+\norm{u_j-P_{V^\perp}(u_j)}^2}\leq \dfrac{d_{0,V}}{\sqrt{k}}
			\sqrt{\dfrac{1+\sigma}{1-\sigma}},\quad  
			\varepsilon_j\leq \dfrac{\sigma^2 d_{0,V}^{\,2}}{2(1-\sigma^2) k}\;;
     \end{aligned}
		\end{align}
	 \item[\emph{(b)}]
	  for any $k\geq 1$,
		 \begin{align}
		  \label{eq:703}
			 \begin{aligned}
      &u_k^a\in T^{\varepsilon_k^{\,a}}(\tilde x_k^a),\\
      &\sqrt{\norm{\tilde x_k^a-P_{V}(\tilde x_k^a)}^2+\norm{u_k^a-P_{V^\perp}(u_k^a)}^2}\leq \dfrac{2d_{0,V}}{k},\quad
      0\leq \varepsilon_k^{\,a}\leq \dfrac{2(1+\sigma/\sqrt{1-\sigma^2})d_{0,V}^{\,2}}{k}\,.
      \end{aligned}
		 \end{align}
	\end{enumerate}
\end{theorem}
 
Next result, which is a direct consequence of Theorem \ref{pr:rate.spin2}(b), gives the iteration-complexity of Algorithm \ref{spin}
to find $x,u\in \HH$ and $\varepsilon>0$ satisfying the termination criterion
\eqref{eq:inc.v02}.

\begin{theorem}\emph{(Iteration-complexity)}
 \label{th:ic.main}
  Let $d_{0,V}$ denote the distance of $x_0$ to the solution set
	\eqref{def:sol.v} and let $\rho,\epsilon>0$ be given torelances.
	Then, \emph{Algorithm \ref{spin}} finds $x,u\in \HH$ and $\varepsilon>0$
	satisfying the termination criterion \eqref{eq:inc.v02}
	in at most
	\begin{align}
	 \label{eq:ic.main02}
	  \mathcal{O}\left(\max\left\{\left\lceil\dfrac{d_{0,V}}{\rho}\right\rceil,
		\left\lceil\dfrac{d_{0,V}^{\,2}}{\epsilon}\right\rceil\right\}\right)
	\end{align}
	iterations.	
\end{theorem}

We now consider a special instance of Algorithm \ref{spin}
which will be used in Section \ref{sec:spin} to derive
operator splitting methods for solving the problem of
finding zeroes of a sum of finitely many maximal monotone operators.

\mgap
\mgap

\noindent
\fbox{
\addtolength{\linewidth}{-2\fboxsep}%
\addtolength{\linewidth}{-2\fboxrule}%
\begin{minipage}{\linewidth}
\begin{algorithm}
\label{spin2}
{\bf An inexact Spingarn's partial inverse method for \bf{(\ref{eq:inc.v})} (II) }
\end{algorithm}
\begin{itemize}
\item[(0)] Let $x_0\in \HH$ and $\sigma\in [0,1[$ be given and set $k=1$.
\item [(1)] Compute $\tilde x_k\in \HH$ and $\varepsilon_k\geq 0$ such that
\begin{align}
\label{eq:spin3}
 \begin{aligned}
  u_k:=x_{k-1}-\tilde x_k\in T^{\varepsilon_k}(\tilde x_k),\quad 
	\varepsilon_k\leq \dfrac{\sigma^2}{2} \norm{\tilde x_k-P_V(x_{k-1})}^2.
	\end{aligned}
\end{align} 
\item[(2)] Define 
 \begin{align}
  \label{eq:spin4}
     x_k= P_V(\tilde x_{k})+P_{V^\perp}(u_k),
   \end{align}
	set $k\leftarrow k+1$ and go to step 1.
   \end{itemize}
\noindent
\end{minipage}
} 
\mgap
\mgap

\noindent
\emph{Remarks.} 
1) Letting $\sigma=0$ in Algorithm \ref{spin2} and using Proposition \ref{pr:teps}(d)
we obtain from \eqref{eq:spin3} that $x=\tilde x_k$ solves the inclusion 
$0\in T(x)+x-x_{k-1}$, i.e., $\tilde x_k=(T+I)^{-1}x_{k-1}$ for all $k\geq 1$.  
In other words, if $\sigma=0$, then Algorithm \ref{spin2} is the Spingarn's partial inverse
method originally presented in \cite{spi-par.amo83}.
2) It follows from Proposition \ref{pr:teps}(e) that Algorithm \ref{spin2} is a generalization to the general setting of inclusions with
monotone operators of the \emph{Epsilon-proximal decomposition method scheme (EPDMS)}
proposed and studied in \cite{ouo-eps.mp04} for solving convex optimization problems.
Indeed, using the identity in \eqref{eq:spin3} we find that the right hand side of
the inequality in 
\eqref{eq:spin3} is equal to $\sigma^2/2\left(\norm{P_{V^\perp}(\tilde x_k)}^2+\norm{P_V(u_k)}^2\right)$
(cf. EPDMS method in \cite{ouo-eps.mp04}, with a different notation).
%
We also mention that no iteration-complexity analysis was performed in \cite{ouo-eps.mp04}.
%
3) Likewise, letting $V=\HH$ in Algorithm \ref{spin2} and using 
Proposition \ref{pr:teps}(e) we obtain that Algorithm \ref{spin2} generalizes the \emph{IPP-CO framework}
of \cite{mon.sva-icc.pre} (with $\lambda_k:=1$ for all $k\geq 1$),
for which iteration-complexity analysis was presented in the latter reference,
 to the more general setting
of inclusions problems with monotone operators.

\begin{proposition}
 \label{th:4.e.3}
The following statements hold true. 
\begin{itemize}
 \item[\emph{(a)}] \emph{Algorithm \ref{spin2}} is a special instance of \emph{Algorithm \ref{spin}}.
\item[\emph{(b)}] The conclusions of Theorem \ref{pr:rate.spin2}
and Theorem \ref{th:ic.main}  are still valid with \emph{Algorithm \ref{spin}} replaced by \emph{Algorithm \ref{spin2}}.
\end{itemize}
\end{proposition}
\begin{proof}
(a) Let $\{x_k\}$, $\{\tilde x_k\}$, $\{\varepsilon_k\}$ and $\{u_k\}$ be generated by Algorithm \ref{spin2}.
Firstly, note that the identity in \eqref{eq:spin3} yields $u_k+\tilde x_k-x_{k-1}=0$ and, consequently,
\begin{align*}
 \norm{\tilde x_k-P_V(x_{k-1})}^2&=\norm{P_V(\tilde x_k-x_{k-1})}^2+\norm{P_{V^\perp}(\tilde x_k)}^2\\
    &=\norm{P_V(\tilde x_k-x_{k-1})}^2+\norm{P_{V^\perp}(u_k-x_{k-1})}^2\\
		&=\norm{P_V(\tilde x_k)+P_{V^\perp}(u_k)-x_{k-1}}^2,
\end{align*}
and
\begin{align*}
 P_V(\tilde x_k)+P_{V^\perp}(u_k)=& (\tilde x_k-P_{V^\perp}(\tilde x_k))+P_{V^\perp}(u_k)\\
       &=(x_{k-1}-u_k)-P_{V^\perp}(\tilde x_k)+P_{V^\perp}(u_k)\\
			 &=x_{k-1}-\left[P_V(u_k)+P_{V^\perp}(\tilde x_k)\right].
\end{align*}
Altogether we obtain (a).

(b) This Item is a direct consequence of (a), Theorem \ref{pr:rate.spin2} and Theorem \ref{th:ic.main}.
\qed
\end{proof}

Next we observe that Proposition \ref{th:4.e.3}(b) and the first remark after Algorithm \ref{spin2}
allow us to obtain the iteration-complexity for the Spingarn's partial inverse method.

\begin{proposition}
 \label{pr:s23}
 Let $d_{0,V}$ denote the distance of $x_0$ to the solution set \eqref{def:sol.v}
 and consider \emph{Algorithm \ref{spin2}} with $\sigma=0$ or, equivalently, the \emph{Spingarn's partial inverse method
of \cite{spi-par.amo83}}. For given tolerances $\rho, \epsilon>0$, the latter method finds
 \begin{itemize}
  \item[\emph{(a)}] 
	$x,u\in \HH$ 
	such that $u\in T(x)$, $\max\left\{\norm{x-P_V(x)},\norm{u-P_{V^\perp}(u)}\right\} \leq \rho$ 
	 in at most
	 \begin{align}
	 \label{eq:ic.main03}
	 \mathcal{O}\left(\left\lceil \dfrac{d_{0,V}^{\,2}}{\rho^2} \right\rceil\right)
	\end{align}
	 iterations.
	\item[\emph{(b)}] 
	$x,u\in \HH$
	and $\varepsilon>0$ satisfying the termination criterion \eqref{eq:inc.v02}
	in at most a number of iterations given in \eqref{eq:ic.main02}.
\end{itemize} 
\end{proposition}
\begin{proof}
(a) The statement in this item is a direct consequence of Proposition \ref{th:4.e.3}(b),
Theorem \ref{pr:rate.spin2}(a) and the fact that $\varepsilon_k=0$ for all $k\geq 1$ 
(because $\sigma=0$
in \eqref{eq:spin3}). (b) Here, the result follows from Proposition \ref{th:4.e.3}(b)
and Theorem \ref{th:ic.main}.
\qed
\end{proof}

\subsection{Proof of Theorem \ref{pr:rate.spin2}}
\label{subsec:pp}

The approach adopted in the current section for solving \eqref{eq:inc.v} 
follows the Spingarn's approach
\cite{spi-par.amo83} which consists
in solving the monotone inclusion 
\begin{align}
 \label{eq:inc.ig}
 0\in T_V(z)
\end{align}
where the maximal monotone operator $T_V:\HH\tos \HH$ is the partial inverse of $T$ with respect to 
the subspace $V$. 
In view of \eqref{eq:def.ig}, we have 
\begin{align}
 \label{eq:id.set}
  (T_V)^{-1}(0)=\mathcal{S}^*(V,T), 
\end{align}
	where the latter set is defined in \eqref{def:sol.v}. Hence, problem \eqref{eq:inc.v} is equivalent to the monotone inclusion problem \eqref{eq:inc.ig}.
%
Before proving Theorem \ref{pr:rate.spin2} we will show that
Algorithm \ref{spin} can be regarded as a special instance of Algorithm \ref{hpe}
for solving \eqref{eq:inc.ig}.

\begin{proposition}
 \label{pr:pi.e.hpe}
 Let $\{\tilde x_k\}_{k\geq 1}$, $\{u_k\}_{k\geq 1}$, $\{\varepsilon_k\}_{k\geq 1}$
and $\{z_k\}_{k\geq 0}$ be generated by \emph{Algorithm \ref{spin}}.
Define $z_0=x_0$ and, for all $k\geq 1$,
  \begin{align}
	  \label{eq:pi.e.hpe}
		 z_k=x_k,\quad \tilde z_k=P_V(\tilde x_k)+P_{V^\perp}(u_k),\quad  v_k=P_V(u_k)+P_{V^\perp}(\tilde x_k).
		 \end{align}
Then, for all $k\geq 1$,
\begin{align}
 \begin{aligned}
 \label{eq:pi.e.hpe02}
 &v_k \in \left(T_V\right)^{\varepsilon_k}(\tilde z_k),\quad \norm{v_k+\tilde z_k-z_{k-1}}^2+
 2\varepsilon_k\leq \sigma^2\norm{\tilde z_k-z_{k-1}}^2,\\
 & z_k=z_{k-1}-v_k,
 \end{aligned} 
\end{align}
i.e., $(\tilde z_k, v_k,\varepsilon_k)$ and $\lambda_k:=1$ satisfy \eqref{eq:hpe} and \eqref{eq:hpe2} for all $k\geq 1$.
As a consequence, the sequences $\{z_k\}_{k\geq 0}$, $\{\tilde z_k\}_{k\geq 1}$, $\{v_k\}_{k\geq 1}$
and $\{\varepsilon_k\}_{k\geq 1}$ are generated by \emph{Algorithm \ref{hpe}}
(with $\lambda_k:=1$ for all $k\geq 1$) for solving \eqref{eq:inc.ig}.
\end{proposition}
\begin{proof}
From the inclusion in \eqref{eq:spin}, \eqref{eq:def.ig} with $S=T^{\varepsilon_k}$
and Lemma \ref{lm:bss} we have 
$P_V(u_k)+P_{V^\perp}(\tilde x_k)\in \left(T_V\right)^{\varepsilon_k}(P_V(\tilde x_k)+P_{V^\perp}(u_k))$
for all $k\geq 1$, which combined with the definitions of $\tilde z_k$ and $v_k$ in \eqref{eq:pi.e.hpe}
gives the inclusion in \eqref{eq:pi.e.hpe02}. Direct use of \eqref{eq:pi.e.hpe} and
the definition of $\{z_k\}$ yield
\begin{align}
\begin{aligned}
 & v_k+\tilde z_k+z_{k-1}=u_k+\tilde x_k-x_{k-1},\\
 & \tilde z_{k}-z_{k-1}=P_V(\tilde x_k)+P_{V^\perp}(u_k)-x_{k-1},\\
 & z_{k-1}-v_k=x_{k-1}-[P_V(u_k)-P_{V^\perp}(\tilde x_k)],
 \end{aligned}
\end{align}
which combined with \eqref{eq:spin}, \eqref{eq:spin2}
and the definition of $\{z_k\}$ gives the remaining statements
in \eqref{eq:pi.e.hpe02}. The last statement of the proposition follows from
\eqref{eq:pi.e.hpe02} and Algorithm \ref{hpe}'s definition.
\qed
\end{proof}

\noindent
{\it Proof of Theorem \ref{pr:rate.spin2}.}
From \eqref{eq:pi.e.hpe} we obtain
\begin{align}
 \label{eq:xzv}
 \tilde x_k=P_V(\tilde z_k)+P_{V^\perp}(v_k),\quad  
u_k=P_V(v_k)+P_{V^\perp}(\tilde z_k)\quad \forall k\geq 1.
 \end{align}
Direct substitution of the latter identities in $\tilde x_k^a$
and $u_k^a$ in \eqref{eq:def.erg2} yields
\begin{align}
 \label{eq:xzv2}
\tilde x_k^a=P_V(\tilde z_k^a)+P_{V^\perp}(v_k^a),\quad  
u_k^a=P_V(v_k^a)+P_{V^\perp}(\tilde z_k^a)\quad \forall k\geq 1.
\end{align}
Using \eqref{eq:xzv} and \eqref{eq:xzv2} in the definition of $\varepsilon_k^a$
in \eqref{eq:def.erg2} and the fact that the operators $P_V$ and $P_{V^\perp}$
are self-adjoint and idempotent we find
\begin{align}
 \label{eq:carc.e}
  \varepsilon_k^{\,a}=
    \frac{1}{\;\Lambda_k}\;\sum_{\ell=1}^k\,\lambda_\ell (\varepsilon_\ell
    +\inner{\tilde {z}_\ell-
     \tilde {z}_k^{\,a}}{v_\ell-v_k^{\,a}})\quad \forall k\geq 1,
\end{align}
where $\{\varepsilon_k^a\}$ is defined in \eqref{eq:def.erg2}. Now 
consider the ergodic sequences $\{\Lambda_k\}$, $\{\tilde z_k^a\}$ and $\{v_k^a\}$
defined in \eqref{eq:d.eg} with $\lambda_k:=1$ for all $k\geq 1$. 
Let $d_0$ denote the distance of $z_0=x_0$ to the solution
set $(T_V)^{-1}(0)$ of \eqref{eq:inc.ig}
and note that $d_0=d_{0,V}$ in view of \eqref{eq:id.set}.
Based on the above considerations one can use
the last statement in Proposition \ref{pr:pi.e.hpe}
 and Theorem \ref{lm:rhpe2} with $\underline{\lambda}:=1$ to conclude that
for any $k\geq 1$ there exists $j\in \{1,\dots,k\}$ 
  such that  
		\begin{align}
		  \label{eq:700}
      v_j\in (T_V)^{\varepsilon_j}(\tilde z_j),\quad  \norm{v_j}\leq \dfrac{d_{0,V}}{\sqrt{k}}
			\sqrt{\dfrac{1+\sigma}{1-\sigma}},\quad  \varepsilon_j\leq \dfrac{\sigma^2 d_{0,V}^{\,2}}{2(1-\sigma^2)k}\,,
    \end{align}
	and
		 \begin{align}
		 \label{eq:701}
     v_k^a\in (T_V)^{\varepsilon_k^a}(\tilde z_k^a),\quad 
     \norm{v_k^a}\leq \dfrac{2d_{0,V}}{k},\quad
     \varepsilon_k^a\leq \dfrac{2(1+\sigma/\sqrt{1-\sigma^2})d_{0,V}^{\,2}}{k}\,,
     \end{align}
where $\{\varepsilon_k^a\}$ is given in \eqref{eq:def.erg2}.
Using Lemma \ref{lm:bss}, the definition in \eqref{eq:def.ig} (for $S=T^{\varepsilon_k}$),
\eqref{eq:xzv} and \eqref{eq:xzv2}
we conclude that the equivalence
%
	$ 
	v\in (T_V)^{\varepsilon}(\tilde z)\iff v\in (T^{\varepsilon})_{V}(\tilde z)
	 \iff u\in T^{\varepsilon}(\tilde x)
	$
%
holds for $(\tilde z,v,\varepsilon)=(\tilde z_k,v_k,\varepsilon_k)$
and $(\tilde x,u,\varepsilon)=(\tilde x_k,u_k,\varepsilon_k)$,
and $(\tilde z,v,\varepsilon)=(\tilde z_k^a,v_k^a,\varepsilon_k^a)$
and  $(\tilde x,u,\varepsilon)=(\tilde x_k^a,u_k^a,\varepsilon_k^a)$, for all $k\geq 1$.
As a consequence, the inclusions in \eqref{eq:702}
	and \eqref{eq:703} follow from the ones 
	in \eqref{eq:700} and \eqref{eq:701}, respectively.
	%
	%
	%
	Since \eqref{eq:xzv2} gives $v_k^a=P_V(u_k^a)+P_{V^\perp}(\tilde x_k^a)$ for all $k\geq 1$,
	it follows from the definition of $\{v_k\}$ in \eqref{eq:pi.e.hpe}
	that $(v,u,\tilde x)=(v_k,u_k,\tilde x_k)$
	and $(v,u,\tilde x)=(v_k^a,u_k^a,\tilde x_k^a)$
	satisfy
	\begin{align*}
	   \norm{v}^2=\norm{P_V(u)}^2+\norm{P_{V^\perp}(\tilde x)}^2=
		\norm{u-P_{V^\perp}(u)}^2+\norm{\tilde x-P_{V}(\tilde x)}^2
		\end{align*}
	for all $k\geq 1$,
	%
	%
 which, in turn, gives that the inequalities in \eqref{eq:702} and \eqref{eq:703} follow
from the ones in \eqref{eq:700} and \eqref{eq:701}, respectively. This concludes the proof.\qed

\section{Applications to Operator Splitting and Optimization}
\label{sec:spin}
In this section we consider the problem of finding $x\in \mathcal{H}$ such that
\begin{align}
 \label{eq:sum10}
 0\in \sum_{i=1}^m T_i(x)
\end{align}
where $m\geq 2$ and $T_i:\HH\tos \HH$ is maximal monotone for $i=1,\dots, m$.
As observed in~\cite{spi-par.amo83}, $x\in \HH$ satisfies the inclusion \eqref{eq:sum10}
if and only if there exist $u_1,\dots, u_m\in \HH$ such that
\begin{align}
 \label{eq:def.s3}
 u_i\in T_i(x)\;\;\mbox{and}\;\;\sum_{i=1}^m\,u_i=0.
\end{align}
That said,  we consider
the (extended) solution set of \eqref{eq:sum10} -- which we assume nonempty -- to be defined by
 \begin{equation}
   \label{def:sol.s}
  \mathcal{S}^{*}(\Sigma):=\left\{(z_i)_{i=1}^m\in \HH^m\;:\;\exists\; x,u_1,u_2,\dots, u_m\in \HH\;
	 \mbox{satisfying (\ref{eq:def.s3})};\;\; z_i=x+u_i\; \forall i=1,\dots, m\right\}.
\end{equation} 
%
%
Due to its importance in solving large-scale problems,
numerical schemes for solving \eqref{eq:sum10} use 
information of each $T_i$ individually instead of
using the entire sum~\cite{spi-par.amo83,bur.sag.sch-ine.oms06,bau.bot.har-att.jat12,bri.com-skew.sjo11,eck.sva-pro.mp08,eck.sva-gen.sjco09}.
In this section, we apply the results of Section \ref{sec:pi} to present and
 study the iteration-complexity of an inexact-version
of the Spingarn's operator splitting method~\cite{spi-par.amo83} for solving \eqref{eq:sum10}
and, as a by-product, we obtain the iteration-complexity of the latter method.
Moreover, we will apply our results to obtain the iteration-complexity
of a parallel forward-backward algorithm for solving multi-term composite
convex optimization problems.

To this end, we consider the following notion of approximate solution for \eqref{eq:sum10}:
given tolerances $\rho, \delta, \epsilon>0$, find $\bar x_1,\bar x_2,\dots, \bar x_m\in \HH$,
$\bar u_1,\bar u_2,\dots, \bar u_m\in \HH$ and $\bar \varepsilon_1, \bar \varepsilon_2, \dots,
\bar\varepsilon_m>0$ such that 
$(x_i)_{i=1}^m=(\bar x_i)_{i=1}^m$, $(u_i)_{i=1}^m=(\bar u_i)_{i=1}^m$
and $(\varepsilon_i)_{i=1}^m=(\bar \varepsilon_i)_{i=1}^m$
satisfy
\begin{align}
 \label{eq:inc.v03}
 \begin{aligned}
 & u_i\in T_i^{\varepsilon_i}(x_i)\quad \forall i=1,\dots, m,\\
 & \left\|\sum_{i=1}^m\,u_i\right\|\leq \rho,\\
 & \norm{x_i-x_\ell}\leq \delta \quad \forall i,\ell=1,\dots, m,\\
 & \sum_{i=1}^m\,\varepsilon_i\leq \epsilon.
\end{aligned}
\end{align}
For $\rho=\delta=\epsilon=0$, criterion \eqref{eq:inc.v03} gives
$\bar x_1=\bar x_2=\cdots=\bar x_m=:\bar x$, $\sum_{i=1}^m\,\bar u_i=0$ and $\bar u_i \in T_i(\bar x)$
for all $i=1,\dots, m$, i.e., in this case
$\bar x, \bar u_1,\bar u_2,\dots, \bar u_m$ satisfy \eqref{eq:def.s3}. 
%

We next present our inexact version of the Spingarn's operator splitting
method \cite{spi-par.amo83} for solving \eqref{eq:sum10}.

\mgap
\mgap

\noindent
\fbox{
\addtolength{\linewidth}{-2\fboxsep}%
\addtolength{\linewidth}{-2\fboxrule}%
\begin{minipage}{\linewidth}
\begin{algorithm}
\label{sum}
{\bf An inexact Spingarn's operator splitting method for \bf{(\ref{eq:sum10})}}
\end{algorithm}
\begin{itemize}
\item[(0)] Let $(x_0, y_{1,0},\dots,y_{m,0})\in \HH^{m+1}$
  such that $y_{1,0}+\dots+y_{m,0}=0$ and $\sigma\in [0,1[$ be given and set $k=1$.
\item [(1)] For each $i=1,\dots, m$, compute $\tilde x_{i,\,k} \in \HH$
            and $\varepsilon_{i,\,k}\geq 0$ such that
\begin{align}
\label{eq:sum}
 \begin{aligned}
  u_{i,\,k}:=x_{k-1}+y_{i,\,k-1}-\tilde x_{i,\,k} \in T_i^{\,\varepsilon_{i,\,k}}(\tilde x_{i,\,k}),\quad
	\quad	\varepsilon_{i,\,k} \leq \frac{ \sigma^2}{2} \norm{\tilde x_{i,\,k}-x_{k-1}}^2.
 \end{aligned}
\end{align} 
\item[(2)] Define 
     \begin{align}
		   \label{eq:sum2}
			 x_k=\dfrac{1}{m}\sum_{i=1}^m\,\tilde x_{i,\,k}, \qquad
			 y_{i,\,k}=u_{i,\,k}-\dfrac{1}{m}\sum_{\ell=1}^m\, u_{\ell,\,k}\;\;\mbox{for}\;\;i=1,\dots, m,
		 \end{align}
 set $k\leftarrow k+1$ and go to step 1.
\end{itemize}
\noindent
\end{minipage}
} 
\mgap
\mgap

\noindent
\emph{Remarks.} 1) Letting $\sigma=0$ in Algorithm \ref{sum} we obtain the
Spingarn's operator splitting method of \cite{spi-par.amo83}. 
2) In \cite{bur.sag.sch-ine.oms06}, Section 5, an inexact version of the Spingarn's operator
splitting method -- called \emph{split-sPIM($\varepsilon$)} -- was proposed 
for solving \eqref{eq:sum10}. With a different notation, for $i=1,\dots, m$, each
iteration of the latter method can be written as:
\begin{align}
\label{eq:68}
\left\{
       \begin{array}{ll}
			  \vspace{0.3cm}
				 u_{i,\,k}\in T_i^{\,\varepsilon_{i,\,k}}(\tilde x_{i,\,k}),\\
				 \norm{u_{i,\,k}+\tilde x_{i,\,k}-x_{k-1}-y_{i,\,k-1}}^2+2\varepsilon_{i,\,k}\leq 
				 \hat\sigma^2 \left(\norm{\tilde x_{i,\,k}-x_{k-1}}^2+\norm{u_{i,\,k}-y_{i,\,k-1}}^2\right),\\
x_k=x_{k-1}-\dfrac{1}{m}\sum_{i=1}^m\,u_{i,\,k},\quad y_{i,\,k}=y_{i,\,k-1}-\tilde x_{i,\,k}
+\dfrac{1}{m}\sum_{\ell=1}^m\tilde x_{\ell,\,k}\;\;\mbox{for}\;\;i=1,\dots, m,\\
        \end{array}
        \right.
\end{align}
where $\hat\sigma\in [0,1[$. 
The convergence analysis of \cite{bur.sag.sch-ine.oms06}
consists in analyzing \eqref{eq:68} in the framework
of the method described in \eqref{eq:v.hpe}, whose 
convergence may fail if we take $\hat\sigma>1/\sqrt{5}$,
as we observed in Remark \ref{rm:notconv}. 
On the other hand, we will prove in Proposition \ref{pr:sum.e.4} that Algorithm \ref{sum}
can be regarded as a special instance of Algorithm \ref{spin2},
which converges for all $\sigma\in [0,1[$ (see Proposition \ref{pr:pi.e.hpe}, Proposition \ref{th:4.e.3}(b)
and \cite[Theorem 3.1]{sol.sva-hyb.svva99}). Moreover, we mention that contrary to this work no iteration-complexity analysis 
is performed in~\cite{bur.sag.sch-ine.oms06}.
%

For each $i=1,\dots, m$, let $\{\tilde x_{i,\,k}\}$, $\{u_{i,\,k}\}$ and $\{\varepsilon_{i,\,k}\}$ be generated by Algorithm \ref{sum}
and define the \emph{ergodic} sequences associated to them:
\begin{align}
  \label{eq:def.erg4}
 \begin{aligned}
 &\tilde x_{i,\,k}^a:=\dfrac{1}{k}\sum_{\ell=1}^k \tilde x_{i,\ell}\,,\quad  u_{i,\,k}^a:=\dfrac{1}{k}\sum_{\ell=1}^k u_{i,\ell}\,,\\ 
 &\varepsilon_{i,\,k}^{\,a}:=\dfrac{1}{k}\sum_{\ell=1}^k\,
\big[\varepsilon_{i,\ell}+\inner{\tilde x_{i,\ell}-\tilde x_{i,\,k}^a}{u_{i,\ell}-u_{i,\,k}^a}
\big].
\end{aligned}
\end{align}


Next theorem will be proved in Subsection \ref{subsec:ppp}.

\begin{theorem}
 \label{pr:rate.spin3}
For each $i=1,\dots, m$, let $\{\tilde x_{i,\,k}\}$, $\{u_{i,\,k}\}$ and $\{\varepsilon_{i,\,k}\}$ be generated by \emph{Algorithm \ref{sum}} and let
 $\{\tilde x_{i,\,k}^a\}$, $\{u_{i,\,k}^a\}$ and $\{\varepsilon_{i,\,k}^{\,a}\}$ be defined in \eqref{eq:def.erg4} . 
Let also $d_{0,\Sigma}$ denote the distance of $(x_0+y_{1,0},\dots, x_0+y_{m,0})$ to the solution set \eqref{def:sol.s}.
 The following statements hold:
\begin{enumerate}
  \item[\emph{(a)}] For any $k\geq 1$,
	  there exists $j\in \{1,\dots,k\}$ 
    such that  
		\begin{align}
		  \label{eq:706}
			\begin{aligned}
      & u_{i,j}\in T_i^{\varepsilon_{i,j}}(\tilde x_{i,j})\quad \forall i=1,\dots,m,\\  
			&\left\|\sum_{i=1}^m\,u_{i,j}\right\|\leq \dfrac{\sqrt{m}\,d_{0,\Sigma}}{\sqrt{k}}
			\sqrt{\dfrac{1+\sigma}{1-\sigma}},\\
			&\norm{\tilde x_{i,j}-\tilde x_{\ell, j}}\leq \dfrac{2\,d_{0,\Sigma}}{\sqrt{k}}
			\sqrt{\dfrac{1+\sigma}{1-\sigma}}\quad \forall i,\ell=1,\dots, m,\\
		 &\sum_{i=1}^m\,\varepsilon_{i,j}\leq \dfrac{\sigma^2 d_{0,\Sigma}^{\,2}}{2(1-\sigma^2) k}\;;
     \end{aligned}
		\end{align}
	 \item[\emph{(b)}]
	  for any $k\geq 1$,
		 \begin{align}
		  \label{eq:707}
			 \begin{aligned}
      &u_{i,\,k}^a\in T_i^{\varepsilon_{i,\,k}^{\,a}}(\tilde x_{i,\,k}^a)\quad \forall i=1,\dots, m,\\
      &\left\|\sum_{i=1}^m\,u_{i,\,k}^a\right\|\leq \dfrac{2\sqrt{m}\,d_{0,\Sigma}}{k},\\
			&\norm{\tilde x_{i,k}^a-\tilde x_{\ell, k}^a}\leq \dfrac{4\,d_{0,\Sigma}}{k}
			\quad \forall i,\ell=1,\dots, m,\\
      &\sum_{i=1}^m\, \varepsilon_{i,\,k}^{\,a}\leq \dfrac{2(1+\sigma/\sqrt{1-\sigma^2})d_{0,\Sigma}^{\,2}}{k}\,.
      \end{aligned}
		 \end{align}
	\end{enumerate}
\end{theorem}

As a consequence of Theorem \ref{pr:rate.spin3}(b)
we obtain the iteration-complexity of Algorithm \ref{sum}
to find $x_1,x_2,\dots, x_m\in \HH$, $u_1,u_2,\dots, u_m\in \HH$  
and $\varepsilon_1,\varepsilon_2,\dots, \varepsilon_m>0$ satisfying the termination criterion
\eqref{eq:inc.v03}.

\begin{theorem}\emph{(Iteration-complexity)}
 \label{th:ic.main4}
  Let $d_{0,\Sigma}$ denote the distance of $(x_0+y_{1,0},\dots, x_0+y_{m,0})$ to the solution set \eqref{def:sol.s} 
	and let $\rho,\delta,\epsilon>0$ be given torelances.
	Then, \emph{Algorithm \ref{sum}} finds $x_1,x_2,\dots, x_m\in \HH$, $u_1,u_2,\dots, u_m\in \HH$  
and $\varepsilon_1,\varepsilon_2,\dots, \varepsilon_m>0$
	satisfying the termination criterion \eqref{eq:inc.v03}
	in at most
	\begin{align}
	 \label{eq:ic.main08}
	  \mathcal{O}\left(\max\left\{\left\lceil\dfrac{\sqrt{m}\,d_{0,\Sigma}}{\rho}\right\rceil, \left\lceil \dfrac{d_{0,\Sigma}}{\delta}
		\right\rceil,\left\lceil\dfrac{d_{0,\Sigma}^{\,2}}{\epsilon}\right\rceil\right\}\right)
	\end{align}
	iterations.	
\end{theorem}

Using the first remark after Algorithm \ref{sum}
and Theorem \ref{pr:rate.spin3} we also obtain 
the \emph{pointwise} and \emph{ergodic} iteration-complexity of Spingarn's operator splitting method~\cite{spi-par.amo83}.

\begin{theorem}\emph{(Iteration-complexity)}
 \label{pr:s24}
 Let $d_{0,\Sigma}$ denote the distance of $(x_0+y_{1,0},\dots, x_0+y_{m,0})$ to the solution set \eqref{def:sol.s}
 and consider \emph{Algorithm \ref{sum}} with $\sigma=0$ or, equivalently, the \emph{Spingarn's operator splitting method
of \cite{spi-par.amo83}}. For given tolerances $\rho,\delta, \epsilon>0$, the latter method finds
 \begin{itemize}
  \item[\emph{(a)}] 
	$x_1,x_2,\dots, x_m\in \HH$ and $u_1,u_2,\dots, u_m\in \HH$   
	such that 
	\begin{align}
	 \label{eq:sos}
	 \begin{aligned}
	 &u_i\in T_i(x_i)\quad \forall i=1,\dots, m,\\
	 &\left\|\sum_{i=1}^ m\,u_i\right\|\leq \rho,\\
	 &\norm{x_i-x_\ell}\leq \delta\,, \quad \forall i,\ell=1,\dots, m,	
	 \end{aligned}
	 \end{align}
	 in at most
	 \begin{align}
	 \label{eq:ic.main07}
	  \mathcal{O}\left(\max\left\{\left\lceil \dfrac{m\,d_{0,\Sigma}^{\,2}}{\rho^2} 
		\right\rceil, \left\lceil \dfrac{d_{0,\Sigma}^{\,2}}{\delta^2} \right\rceil\right\}\right)
	\end{align}
	 iterations.
	\item[\emph{(b)}] 
	$x_1,x_2,\dots, x_m\in \HH$, $u_1,u_2,\dots, u_m\in \HH$
	and $\varepsilon_1,\varepsilon_2,\dots, \varepsilon_m>0$
	satisfying the termination criterion \eqref{eq:inc.v03}
	in at most the number of iterations given in \eqref{eq:ic.main08}.
\end{itemize} 
\end{theorem}
\begin{proof}
(a) This item follows from Theorem \ref{pr:rate.spin3}(a) and the fact that $\varepsilon_{i,\,k}=0$ 
for each $i=1,\dots, m$ and for all $k\geq 1$ 
(because $\sigma=0$
in \eqref{eq:sum}). (b) This item follows directly from Theorem \ref{th:ic.main4}.
\qed
\end{proof}

	\mgap
	\noindent
	{\bf Applications to optimization.}
	In the remaining part of this section we show how
	Algorithm \ref{sum} and its iteration-complexity
  results can be used to derive a
	\emph{parallel forward-backward splitting method} for multi-term composite convex optimization
	and to study its iteration-complexity.
	More precisely, consider the minimization problem
	\begin{align}
	 \label{eq:mp}
	  \min_{x\in \HH}\,\sum_{i=1}^m \left(f_i+\varphi_i\right)(x)
	\end{align}
	where $m\geq 2$ and the following
	conditions are assumed to hold for all $i=1,\dots, m$:
	\begin{itemize}
	 \item[\mbox{(A.1)}] $f_i:\HH\to \R$ is convex, and differentiable with a $L_i$-Lipschitz continuous gradient,
	 i.e., there exists $L_i>0$ such that
		 \begin{align}
		  \label{eq:lips}
			 \norm{\nabla f_i(x)-\nabla f_i(y)}\leq L_i\norm{x-y}\quad \forall x,y\in \HH;
		 \end{align}
		\item[\mbox{(A.2)}] $\varphi_i:\HH\to \BR$ is proper, convex and closed
		with an easily computable resolvent $(\lambda\partial \varphi_i+I)^{-1}$, for any $\lambda>0$;
	 \item[\mbox{(A.3)}] the solution set of \eqref{eq:mp} is nonempty.
	 \end{itemize}

	We also assume standard regularity conditions
	\footnote{see, e.g., \cite[Corollary 16.39]{bau.com-book}}
	on the functions $\varphi_i$ which make \eqref{eq:mp} 
	equivalent to the monotone inclusion problem \eqref{eq:sum10}
	 with $T_i:=\nabla f_i+\partial \varphi_i$, for all $i=1,\dots, m$, i.e., which make it equivalent
	to the problem of finding $x\in \HH$ such that
	\begin{align}
	 \label{eq:sum11}
	 0\in \sum_{i=1}^m\, \left(\nabla f_i+\partial \varphi_i\right)(x).
	\end{align}

Analogously to \eqref{eq:inc.v03}, we consider the following
notion of approximate solution for \eqref{eq:mp}:
given tolerances $\rho, \delta, \epsilon>0$, find $\bar x_1,\bar x_2,\dots, \bar x_m\in \HH$,
$\bar u_1,\bar u_2,\dots, \bar u_m\in \HH$ and $\bar \varepsilon_1, \bar \varepsilon_2, \dots,
\bar\varepsilon_m>0$ such that 
$(x_i)_{i=1}^m=(\bar x_i)_{i=1}^m$, $(u_i)_{i=1}^m=(\bar u_i)_{i=1}^m$
and $(\varepsilon_i)_{i=1}^m=(\bar \varepsilon_i)_{i=1}^m$
satisfy \eqref{eq:inc.v03} with $T_i^{\varepsilon_i}$ replaced
by $\partial_{\varepsilon_i}\,f_i+\partial \varphi_i$, for each $i=1,\dots, m$.
%
%
For $\rho=\delta=\epsilon=0$, this criterion gives
$\bar x_1=\bar x_2=\cdots=\bar x_m=:\bar x$, $\sum_{i=1}^m\,\bar u_i=0$ and $\bar u_i \in 
\left(\nabla f_i+\partial \varphi_i\right)(\bar x)$
for all $i=1,\dots, m$, i.e., in this case
$\bar x$ solves \eqref{eq:sum11}.

	We will present a parallel forward-backward method for solving \eqref{eq:mp} whose iteration-complexity
	is obtained by regarding it as a special instance of Algorithm \ref{sum}. 
	 Since problem \eqref{eq:mp} appears in
	 various applications of convex optimization, it turns out
	that the development of efficient numerical schemes for solving it -- specially with $m\geq 2$ very
	large -- is of great importance.

\mgap

Next is our method for solving \eqref{eq:mp}.

	\mgap
\mgap

\noindent
\fbox{
\addtolength{\linewidth}{-2\fboxsep}%
\addtolength{\linewidth}{-2\fboxrule}%
\begin{minipage}{\linewidth}
\begin{algorithm}
\label{fb}
{\bf A parallel forward-backward splitting method for \bf{(\ref{eq:mp})}}
\end{algorithm}
\begin{itemize}
\item[(0)] Let $(x_0, y_{1,0},\dots,y_{m,0})\in \HH^{m+1}$
  such that $y_{1,0}+\dots+y_{m,0}=0$ and $\sigma\in ]0,1[$ be given and set 
	$\lambda=\sigma^2/\max\{L_i\}_{i=1}^m$ and $k=1$.
\item [(1)] For each $i=1,\dots, m$, compute 
%
\begin{align}
\label{eq:sum80}
 \begin{aligned}
  &\tilde x_{i,\,k}=(\lambda \partial \varphi_i + I)^{-1}\left(x_{k-1}+y_{i,\,k-1}-\lambda \nabla f_i(x_{k-1})\right).
 \end{aligned}
\end{align} 
\item[(2)] Define 
     \begin{align}
		   \label{eq:sum82}
			 x_k=\dfrac{1}{m}\sum_{i=1}^m\,\tilde x_{i,\,k}, \qquad
		   y_{i,\,k}=y_{i,\,k-1}+x_k-\tilde x_{i,\,k}\;\;\mbox{for}\;\; i=1,\dots, m, 
		\end{align}
 set $k\leftarrow k+1$ and go to step 1.
\end{itemize}
\noindent
\end{minipage}
} 
\mgap
\mgap

\noindent
\emph{Remarks.} 1) 
Since in \eqref{eq:sum80}
we have a forward step in the direction $-\nabla f_i(x_{k-1})$ and a backward step
given by the resolvent of $\varphi_i$, Algorithm \ref{fb} can be regarded as
a parallel variant of the classical forward-backward splitting algorithm \cite{fac-pan.bookI} . 2) For $m=1$ the above method coincides with
the forward-backward method of \cite{mon.sva-icc.pre}, for which the iteration-complexity was
studied in the latter reference.


For each $i=1,\dots, m$, let $\{x_{k}\}$, $\{\tilde x_{i,\,k}\}$
be generated by Algorithm \ref{fb}, $\{u_{i,\,k}\}$ and $\{\varepsilon_{i,\,k}\}$
be defined in \eqref{eq:def.u.fb}  and let
 $\{\tilde x_{i,\,k}^a\}$, $\{u_{i,\,k}^a\}$ and $\{\varepsilon_{i,\,k}^{\,a}\}$ be given in \eqref{eq:def.erg4}.
Define, for all $k\geq 1$,
\begin{align}
\label{eq:1400}
 \begin{aligned}
& u'_{i,\,k}:=\dfrac{1}{\lambda}\,u_{i,\,k}\,,\quad
\varepsilon'_{i,\,k}:=\dfrac{1}{\lambda}\,\varepsilon_{i,\,k}\,,\quad  
u'^{\,a}_{i,\,k}:=\dfrac{1}{\lambda}\,u_{i,\,k}^a\,,\quad 
\varepsilon'^{\,a}_{i,\,k}:=\dfrac{1}{\lambda}\,\varepsilon_{i,\,k}^a\,,\\
& \varepsilon''^{\,a}_{i,\,k}:=\dfrac{1}{k}\sum_{\ell=1}^k\,
\Big[\varepsilon'_{i,\ell}+\inner{\tilde x_{i,\ell}-\tilde x_{i,\,k}^a}{\nabla f_i(x_{\ell-1})-
\dfrac{1}{k}\sum_{s=1}^k\,\nabla f_i(x_{s-1})
}\Big].
\end{aligned}
\end{align}
%
%
%
%

Next theorem will be proved in Subsection \ref{subsec:p8}.

\begin{theorem}
 \label{pr:rate.13}
For each $i=1,\dots, m$, let $\{\tilde x_{i,\,k}\}$
be generated by \emph{Algorithm \ref{fb}}
and $\{\tilde x^{\,a}_{i,\,k}\}$ be given in \eqref{eq:def.erg4}; let
$\{u'_{i,\,k}\}$, $\{\varepsilon'_{i,\,k}\}$, $\{u'^{\,a}_{i,\,k}\}$, $\{\varepsilon'^{\,a}_{i,\,k}\}$
and $\{\varepsilon''^{\,a}_{i,\,k}\}$ be given in \eqref{eq:1400}.
Let also $d_{0,\Sigma}$ denote the distance of $(x_0+y_{1,0},\dots, x_0+y_{m,0})$ to the solution set \eqref{def:sol.s}
in which $T_i:=\nabla f_i+\partial \varphi_i$ for $i=1,\dots, m$, and define $L_{\Sigma}:=\max\{L_i\}_{i=1}^m$.
%
%
%
%
 The following  hold:
\begin{enumerate}
  \item[\emph{(a)}] For any $k\geq 1$,
	  there exists $j\in\{1,\dots,k\}$ 
    such that  
		\begin{align}
		  \label{eq:1206}
			\begin{aligned}
      & u'_{i,j}\in \left(\partial_{\varepsilon'_{i,j}}f_i+\partial \varphi_i\right)(\tilde x_{i,j})\quad \forall i=1,\dots,m,\\  
			& \left\|\sum_{i=1}^m\,u'_{i,j}\right\|\leq \dfrac{\sqrt{m}\, L_\Sigma\, d_{0,\Sigma}}{\sigma^2 \sqrt{k}}
			\sqrt{\dfrac{1+\sigma}{1-\sigma}},\\
			& \norm{\tilde x_{i,j}-\tilde x_{\ell, j}}\leq
        \dfrac{2\,d_{0,\Sigma}}{\sqrt{k}}
			\sqrt{\dfrac{1+\sigma}{1-\sigma}} \quad \forall i,\ell=1,\dots, m,\\
			&\sum_{i=1}^m\,\varepsilon'_{i,j}\leq \dfrac{L_\Sigma\,d_{0,\Sigma}^{\,2}}{2(1-\sigma^2) k}\;;
     \end{aligned}
		\end{align}
	 \item[\emph{(b)}]
	  for any $k\geq 1$,
		 \begin{align}
		  \label{eq:1207}
			 \begin{aligned}
      & u'^{\,a}_{i,\,k}\in \left(\partial_{\varepsilon''^{\,a}_{i,\,k}}f_i+
			\partial_{\left(\varepsilon'^{\,a}_{i,\,k}-\varepsilon''^{\,a}_{i,\,k}\right)}\varphi_i
			\right)(\tilde x^a_{i,\,k})\quad \forall i=1,\dots, m,\\
		  & \left\|\sum_{i=1}^m\,u'^a_{i,k}\right\|\leq \dfrac{2\sqrt{m}\,L_{\Sigma}\,d_{0,\Sigma}}{\sigma^2 k}, \\
			& \norm{\tilde x^a_{i,\,k}-\tilde x^a_{\ell,k}}\leq \dfrac{4d_{0,\Sigma}}{k}\quad \forall i,\ell=1,\dots, m, \\
      &\sum_{i=1}^m\, \varepsilon'^{\,a}_{i,\,k}\leq \dfrac{2(1+\sigma/\sqrt{1-\sigma^2})\,L_{\Sigma}\,d_{0,\Sigma}^{\,2}}{\sigma^2 k}\,.
      \end{aligned}
		 \end{align}
	\end{enumerate}
\end{theorem}

The following theorem is a direct consequence of Theorem \ref{pr:rate.13}.

\begin{theorem}\emph{(iteration-complexity)}
 Let $d_{0,\Sigma}$ denote the distance of $(x_0+y_{1,0},\dots, x_0+y_{m,0})$ to the solution set \eqref{def:sol.s} 
in which $T_i:=\nabla f_i+\partial \varphi_i$, for $i=1,\dots, m$, and let $\rho,\delta,\epsilon>0$ be given torelances.
Let $L_\Sigma:=\max\{L_i\}_{i=1}^m$.
	Then, \emph{Algorithm \ref{fb}} finds
\begin{itemize}
 \item[\emph{(a)}]
 $x_1,x_2,\dots, x_m\in \HH$, $u_1,u_2,\dots, u_m\in \HH$
and $\varepsilon_1,\varepsilon_2,\dots, \varepsilon_m>0$ satisfying
the termination criterion \eqref{eq:inc.v03}
with $T_i^{\varepsilon_i}$ replaced by $\partial_{\varepsilon_i} f_i+\partial\varphi_i$
in at most
\begin{align}
 \mathcal{O}\left(\max\left\{\left\lceil\dfrac{m\,L_{\Sigma}^2\,d_{0,\Sigma}^2}{\rho^2}\right\rceil, 
 \left\lceil \dfrac{d_{0,\Sigma}^2}{\delta^2}
		\right\rceil,
		\left\lceil\dfrac{L_{\Sigma}\,d_{0,\Sigma}^{\,2}}{\epsilon}\right\rceil\right\}\right)
\end{align}
iterations.
\item[\emph{(b)}] 
	$x_1,x_2,\dots, x_m\in \HH$, $u_1,u_2,\dots, u_m\in \HH$, $\varepsilon_1,\varepsilon_2,\dots, \varepsilon_m>0$   
	and $\hat \varepsilon_1,\hat \varepsilon_2,\dots, \hat \varepsilon_m>0$
	such that 
	\begin{align}
	 \begin{aligned}
	 &u_i\in \left(\partial_{\varepsilon_i}f_i+\partial_{\hat\varepsilon_i}\varphi_i\right)(x_i)\quad \forall i=1,\dots, m,\\
	 &\left\|\sum_{i=1}^ m\,u_i\right\|\leq \rho,\\
	 &\norm{x_i-x_\ell}\leq \delta\,, \quad \forall i,\ell=1,\dots, m,\\
	 &\sum_{i=1}^m\,(\varepsilon_i+\hat\varepsilon_i)\leq \epsilon
	 \end{aligned}
	 \end{align}
	 in at most
	 \begin{align}
	  \mathcal{O}\left(\max\left\{\left\lceil \dfrac{\sqrt{m}\,L_\Sigma\,d_{0,\Sigma}}{\rho} 
		\right\rceil, \left\lceil \dfrac{d_{0,\Sigma}}{\delta} \right\rceil,
		\left\lceil\dfrac{L_{\Sigma}\,d_{0,\Sigma}^{\,2}}{\epsilon}\right\rceil
		\right\}\right)
	\end{align}
	 iterations.
	
\end{itemize}
\end{theorem}

\subsection{Proof of Theorem \ref{pr:rate.spin3}}
\label{subsec:ppp}

Analogously to Subsection \ref{subsec:pp}, in the current section
we follow the Spingarn's approach in \cite{spi-par.amo83}
for solving problem \eqref{eq:sum10} which consists
in solving the following inclusion in the product space 
$\HH^m$:
\begin{align}
 {\bf 0}\in {\bf T}_{{\bf V}}({\bf z}),
\end{align} 
where ${\bf T}:\HH^m\tos \HH^m$ is the maximal monotone operator defined by
\begin{align}
 \label{eq:def.t}
  {\bf T}(x_1,x_2,\dots,x_m):=T_1(x_1)\times T_2(x_2)\times \cdots \times T_m(x_m)\quad \forall (x_1,x_2,\dots,x_m)\in \HH^m,
\end{align}
and 
\begin{align}
 \label{eq:def.v}
 {\bf V}:=\left\{(x_1,x_2,\dots,x_m)\in \HH^m\;:\;x_1=x_2=\dots=x_m\right\}
\end{align}
is a closed subspace of $\HH^m$
whose orthogonal complement is
%
\begin{align}
\label{eq:v.perp}
 {\bf V}^\perp=\left\{(x_1,x_2,\dots,x_m)\in \HH^m\;:\;x_1+x_2+\cdots+x_m=0\right\}.
\end{align}
Based on the above observations, we have that problem \eqref{eq:sum10} is equivalent to \eqref{eq:inc.v} 
with ${\bf T}$ and ${\bf V}$ given in \eqref{eq:def.t} and \eqref{eq:def.v}, respectively.
Moreover, in this case, the orthogonal projections onto ${\bf V}$ and ${\bf V}^\perp$ have the
explicit formulae:
\begin{align}
 \label{eq:f.p}
 \begin{aligned}
 &P_{{\bf V}}(x_1,x_2,\dots, x_m)=\left(\dfrac{1}{m}\sum_{i=1}^m\,x_i,\dots,\dfrac{1}{m}\sum_{i=1}^m\,x_i\right),\\
 &P_{{\bf V}^\perp}(x_1,x_2,\dots, x_m)=\left(x_1-\dfrac{1}{m}\sum_{i=1}^m\,x_i,\dots,x_m-\dfrac{1}{m}\sum_{i=1}^m\,x_i\right).
\end{aligned}
\end{align}

Next we show that Algorithm \ref{sum} can be regarded as a special instance
of Algorithm \ref{spin2} and, as a consequence, we will obtain that
Theorem \ref{pr:rate.spin3} follows from results of Section \ref{sec:pi}
for Algorithm \ref{spin2}.

\begin{proposition}
\label{pr:sum.e.4}
 Let $\{x_k\}_{k\geq 0}$ and, for each $i=1,\dots, m$, $\{y_{i,\,k}\}_{k\geq 0}$, $\{\tilde x_{i,\,k}\}_{k\geq 1}$, 
 $\{u_{i,\,k}\}_{k\geq 1}$ and $\{\varepsilon_{i,\,k}\}_{k\geq 1}$
 be generated by \emph{Algorithm \ref{sum}}. 
 Consider the sequences $\{{\bf x}_k\}_{k\geq 0}$, $\{\tilde {\bf x}_k\}_{k\geq 1}$ and $\{{\bf u}_k\}_{k\geq 1}$
 in $\HH^m$ and $\{\varepsilon_k\}_{k\geq 1}$ in $\R_+$ where
\begin{align}
  \label{eq:234}
 \begin{aligned}
 & {\bf x}_k:=(x_k+y_{1,\,k},\dots, x_k+y_{m,\,k}),\quad \tilde {\bf x}_k:=(\tilde x_{1,\,k},\dots, \tilde x_{m,\,k}),\\
 & \varepsilon_k:=\sum_{i=1}^m\,\varepsilon_{i,\,k},\quad {\bf u}_k:=(u_{1,\,k},\dots, u_{m,\,k}).
  \end{aligned}
	\end{align}
	Then, for all $k\geq 1$,
	\begin{align}
	 \label{eq:spin.e.hpe02}
	  \begin{aligned}
	  & {\bf u}_k\in \left(T_1^{\,\varepsilon_{1,k}}\times \dots \times T_m^{\,\varepsilon_{m,k}}\right)(\tilde {\bf x}_k),
		\quad {\bf u}_k+\tilde {\bf x}_k-{\bf x}_{k-1}=0,
		\quad \varepsilon_k\leq \dfrac{\sigma^2}{2}\norm{\tilde {\bf x}_k-P_{\bf V}({\bf x}_{k-1})}^2,\\
		& {\bf x}_k=P_{\bf V}({\bf \tilde x}_{k})+P_{{\bf V}^\perp}({\bf u}_k).
		\end{aligned}
		\end{align}
 As a consequence of \eqref{eq:spin.e.hpe02}, the sequences 
$\{{\bf x}_k\}_{k\geq 0}$, $\{\tilde {\bf x}_k\}_{k\geq 1}$, $\{{\bf u}_k\}_{k\geq 1}$
and $\{\varepsilon_k\}_{k\geq 1}$ are generated by \emph{Algorithm \ref{spin2}}
for solving \eqref{eq:inc.ig} with ${\bf T}$ and ${\bf V}$ given
in \eqref{eq:def.t} and \eqref{eq:def.v}, respectively.		
\end{proposition}
\begin{proof}
Note that \eqref{eq:spin.e.hpe02} follows directly from
\eqref{eq:sum}, \eqref{eq:sum2}, \eqref{eq:234} and definition \eqref{eq:def.prs}
(with $S_i=T^{\varepsilon_{i,k}}$ for $i=1,\dots, m$). The last statement of
the Proposition is a direct consequence of \eqref{eq:spin.e.hpe02}
and Algorithm \ref{spin2}'s definition.
\qed
\end{proof}

\noindent
{\it Proof of Theorem \ref{pr:rate.spin3}}.
We start by defining the ergodic
sequences associated to the sequences $\{\tilde {\bf x}_k\}$,
$\{{\bf u}_k\}$ and $\{\varepsilon_k\}$ in \eqref{eq:234}:
\begin{align}
  \label{eq:def.erg20}
 \begin{aligned}
 &\tilde {\bf x}_k^a:=\dfrac{1}{k}\sum_{\ell=1}^k \tilde {\bf x}_\ell\,,\quad  {\bf u}_k^a:=\dfrac{1}{k}\sum_{\ell=1}^k {\bf u}_\ell\,,\\ 
 &\varepsilon_k^{\,a}:=\dfrac{1}{k}\sum_{\ell=1}^k\,
\big[\varepsilon_\ell+\inner{\tilde {\bf x}_\ell-\tilde {\bf x}_k^a}{{\bf u}_\ell-{\bf u}_k^a}\big].
\end{aligned}
\end{align} 
Note that from \eqref{def:sol.v}, \eqref{def:sol.s}, \eqref{eq:def.t}, \eqref{eq:def.v}
and \eqref{eq:v.perp} we obtain $\mathcal{S}^*({\bf V}, {\bf T})=\mathcal{S}^*(\Sigma)$
and, consequently, $d_{0,{\bf V}}=d_{0,\Sigma}$.
That said, it follows from the last statement in Proposition \ref{pr:sum.e.4}, Proposition 
\ref{th:4.e.3}(a) and Theorem \ref{pr:rate.spin2} that 
for any $k\geq 1$,
there exists $j\in\{1,\dots,k\}$ 
    such that  
		\begin{align}
		  \label{eq:802}
			\begin{aligned}
      & {\bf u}_j\in \left(T_1^{\varepsilon_{1,\,j}}\times T_2^{\varepsilon_{2,\,j}}
			\times \cdots \times T_m^{\varepsilon_{m,\,j}}
			\right)(\tilde {\bf x}_j),\\  
			&\sqrt{\norm{\tilde {\bf x}_j-P_{{\bf V}}(\tilde {\bf x}_j)}^2+\norm{{\bf u}_j-P_{{\bf V}^\perp}({\bf u}_j)}^2}\leq 
			\dfrac{d_{0,\Sigma}}{\sqrt{k}}
			\sqrt{\dfrac{1+\sigma}{1-\sigma}},\quad  
			\varepsilon_j\leq \dfrac{\sigma^2 d_{0,\Sigma}^{\,2}}{2(1-\sigma^2) k}\;,
     \end{aligned}
		\end{align}
		and
	\begin{align}
		  \label{eq:803}
			 \begin{aligned}
      &\sqrt{\norm{\tilde {\bf x}_k^a-P_{{\bf V}}(\tilde {\bf x}_k^a)}^2+\norm{{\bf u}_k^a-P_{{\bf V}^\perp}({\bf u}_k^a)}^2}\leq 
			\dfrac{2d_{0,\Sigma}}{k},\quad
      0\leq \varepsilon_k^{\,a}\leq \dfrac{2(1+\sigma/\sqrt{1-\sigma^2})d_{0,\Sigma}^{\,2}}{k}\,.
      \end{aligned}
		 \end{align}
		In particular, we see that Item (a) of Theorem \ref{pr:rate.spin3}
		follows from \eqref{eq:802}, \eqref{eq:234} and \eqref{eq:f.p}. 
		Note now that from \eqref{eq:def.erg20}, \eqref{eq:234} and \eqref{eq:def.erg4} we obtain, for all $k\geq 1$,
		 \begin{align}
		\label{eq:999}
		 {\bf \tilde x}^a_k=(\tilde x_{1,\,k}^a,\tilde x_{2,\,k}^a,\dots, \tilde x_{m,\,k}^a),\quad 
		 {\bf u}^a_k=(u_{1,\,k}^a,u_{2,\,k}^a,\dots, u_{m,\,k}^a),\quad
		 \varepsilon_k^a=\sum_{i=1}^m\,\varepsilon_{i,\,k}^a.
		\end{align}
		Hence, the inequalities in \eqref{eq:707} follow
		from \eqref{eq:803}, \eqref{eq:999} and \eqref{eq:f.p}.
		To finish the proof of the theorem it suffices to show the inclusions in
		\eqref{eq:707} for each $i=1,\dots, m$ and all $k\geq 1$. To this end,
		note that for each $i=1,\dots, m$ the desired inclusion is a direct consequence of
		the inclusions in \eqref{eq:sum}, the definitions in \eqref{eq:def.erg4}
		and Theorem \ref{th:tf} (with $T=T_i$ for each $i=1,\dots, m$).\qed

\subsection{Proof of Theorem \ref{pr:rate.13}}
\label{subsec:p8}
Next proposition shows that Algorithm \ref{fb} is a special instance of Algorithm \ref{sum} for solving \eqref{eq:sum10}
with $T_i=\nabla (\lambda f_i)+\partial (\lambda \varphi_i)$ for all $i=1,\dots, m$.

\begin{proposition}
 \label{pr:fb.e.sum}
 Let $\{x_k\}_{k\geq 0}$ and, for $i=1,\dots, m$, $\{y_{i,\,k}\}_{k\geq 0}$ and $\{\tilde x_{i,\,k}\}_{k\geq 1}$
 be generated by \emph{Algorithm \ref{fb}}. 
 For $i=1,\dots, m$, consider the sequences $\{u_{i,\,k}\}_{k\geq 1}$
 and $\{\varepsilon_{i,\,k}\}_{k\geq 1}$
 where,  
for all $k\geq 1$,
 \begin{align}
  \label{eq:def.u.fb}
	 \begin{aligned}
	 &u_{i,\,k}:=x_{k-1}+y_{i,\,k-1}-\tilde x_{i,\,k},\\
	 &\varepsilon_{i,\,k}:=\lambda\left[f_i(\tilde x_{i,\,k})-f_i(x_{k-1})-\inner{\nabla f_i(x_{k-1})}
	 {\tilde x_{i,\,k}-x_{k-1}}\right].
	 \end{aligned}
 \end{align} 
Then,  
for all $k\geq 1$,
\begin{align}
\label{eq:sx}
  & \nabla(\lambda f_i)(x_{k-1})\in \partial_{\varepsilon_{i,\,k}} (\lambda f_i)(\tilde x_{i,\,k}),\\
	\label{eq:sx02}
	& u_{i,\,k}-\nabla(\lambda f_i)(x_{k-1})\in \partial(\lambda \varphi_i)(\tilde x_{i,\,k}),\\
	\label{eq:sx03}
	&u_{i,\,k}\in \left(\partial_{\varepsilon_{i,\,k}} (\lambda f_i)
	+\partial (\lambda \varphi_i)\right)(\tilde x_{i,\,k}),\\
	\label{eq:sx04}
  &0\leq \varepsilon_{i,\,k} \leq \frac{\sigma^2}{2} \norm{\tilde x_{i,\,k}-x_{k-1}}^2,\\
	\label{eq:sx05}
	& x_{k}\;\;\mbox{and}\;\;y_{i,\,k}\;\;\mbox{satisfy}\;\; (\ref{eq:sum2}).
\end{align} 
As a consequence of \eqref{eq:def.u.fb}--\eqref{eq:sx05}, the sequences $\{x_{k}\}_{k\geq 0}$, $\{y_{i,\,k}\}_{k\geq 1}$, 
$\{\tilde x_{i,\,k}\}_{k\geq 0}$,
$\{\varepsilon_{i,k}\}_{k\geq 1}$ and $\{u_{i,\,k}\}_{k\geq 1}$ are generated by
\emph{Algorithm \ref{sum}} for solving \eqref{eq:sum10} with
\[
T_i=\nabla (\lambda f_i)+\partial (\lambda \varphi_i)\quad \forall i=1,\dots, m.
\]
 \end{proposition}
\begin{proof}
 Inclusion \eqref{eq:sx} follows from Lemma \ref{lm:ms}
with $(f,x,\tilde x,v,\varepsilon)=(\lambda f_i,x_{k-1}, \tilde x_{i,\,k},\nabla (\lambda f_i)(x_{k-1}), \varepsilon_{i,\,k})$,
where $\varepsilon_{i,\,k}$ is given in \eqref{eq:def.u.fb}. Inclusion \eqref{eq:sx02} follows
from \eqref{eq:sum80}, the first identity in  \eqref{eq:def.u.fb} and Lemma \ref{lm:crules}(a). Inclusion \eqref{eq:sx03}
is a direct consequence of \eqref{eq:sx} and \eqref{eq:sx02}. The inequalities in \eqref{eq:sx04} follow
from assumption (A.1), the second identity in \eqref{eq:def.u.fb}, Lemma \ref{lm:ineq.lin} and the definition of $\lambda>0$
in Algorithm \ref{fb}. The fact that $x_k$ satisfies \eqref{eq:sum2} follows from the first identities
in \eqref{eq:sum2} and \eqref{eq:sum82}. Direct use of \eqref{eq:sum82}
and the assumption that $y_{1,0}+\cdots +y_{m,0}=0$ in step 0 of Algorithm \ref{fb} 
gives $\sum_{\ell=1}^m\,y_{\ell,\,k}=0$ for all $k\geq 0$, which, in turn,
combined with the second identity in \eqref{eq:sum82} and the first identity in \eqref{eq:def.u.fb} proves that
$y_{i,\,k}$ satisfies the second identity in \eqref{eq:sum2}. Altogether, we
obtain \eqref{eq:sx05}. The last statement of the proposition follows
from \eqref{eq:def.u.fb}--\eqref{eq:sx05} and Proposition \ref{pr:teps}(b;\,e).
\qed
\end{proof}

\noindent
{\it Proof of Theorem \ref{pr:rate.13}}.
 From the last statement of Proposition \ref{pr:fb.e.sum}, the fact that
$$\left(\sum_{i=1}^m\,\left[\nabla f_i+\partial \varphi_i\right]\right)^{-1}(0)=
\left(\sum_{i=1}^m\,\left[\nabla (\lambda f_i)+\partial (\lambda \varphi_i)\right]\right)^{-1}(0)
$$
and Theorem \ref{pr:rate.spin3} we obtain that \eqref{eq:706} and \eqref{eq:707} hold. As a consequence of the
latter fact, \eqref{eq:sx03}, \eqref{eq:1400}, Lemma \ref{lm:ineq.lin}(b), the fact
that $\lambda=\sigma^2/L_\Sigma$ and some direct calculations 
we obtain \eqref{eq:1206} and the inequalities in \eqref{eq:1207}. To finish the proof, it suffices
to prove the inclusion in \eqref{eq:1207}. To this end, note first that from \eqref{eq:sx}, \eqref{eq:def.erg4}, 
the last identity in \eqref{eq:1400}, Lemma \ref{lm:crules}(b) and Theorem \ref{th:tf}(b) we obtain, for each
$i=1,\dots, m$,
\begin{align}
 \label{eq:1402}
 \dfrac{1}{k}\sum_{s=1}^k\,\nabla f_i(x_{s-1})\in \partial_{\varepsilon''^{\,a}_{i,\,k}} f_i(\tilde x^{\,a}_{i,\,k})
\quad \forall k\geq 1.
\end{align}
On the other hand, it follows from \eqref{eq:sx02}, Lemma \ref{lm:crules}(a), \eqref{eq:1400}, Theorem \ref{th:tf}(b) and some direct calculations that, for each
$i=1,\dots, m$,
\begin{align}
 \label{eq:1403}
 u'^{\,a}_{i,\,k}-\dfrac{1}{k}\sum_{s=1}^k\,\nabla f_i(x_{s-1})\in 
\partial_{\left(\varepsilon'^{\,a}_{i,\,k}-\varepsilon''^{\,a}_{i,\,k}\right)}\varphi_i (\tilde x^{\,a}_{i,\,k})
\quad \forall k\geq 1,
\end{align}
which, in turn, combined with \eqref{eq:1402} gives the inclusion in \eqref{eq:1207}.
\qed

\section{Conclusions}
\label{sec:cr}
We proposed and analyzed the iteration-complexity of an inexact version
of the Spingar's partial inverse method and, as a consequence, we obtained
the iteration-complexity of an inexact version of the Spingarn's operator
splitting method as well as of a parallel forward-backward method
for multi-term composite convex optimization. We proved that our
method falls in the framework of the hybrid proximal extragradient (HPE) method, 
for which the iteration-complexity has been obtained recently by Monteiro and Svaiter.
We also introduced a notion of approximate solution for the Spingarn's problem
(which generalizes the one introduced by Monteiro and Svaiter for monotone inclusions)
and proved the iteration-complexity for the above mentioned methods based on this notion of
approximate solution.

\section*{Acknowledgments}
The work of S. C. L. was partially supported by CAPES Scholarship no.
201302186. The work of the second author was partially supported by CNPq grants no.
406250/2013-8, 237068/2013-3 and 306317/2014-1.




%
%

\end{document}